\documentclass[3p,times]{elsarticle}
\usepackage{amssymb}
\usepackage{amsmath} 
\usepackage{amsthm}
\usepackage{soul,xcolor}
\usepackage{caption}
\usepackage{subcaption}
\usepackage{graphicx}
\usepackage{tikz}
\usepackage{tkz-euclide}
\usepackage{tikz-dimline}
\usepackage{pgfplots}
\usetikzlibrary{patterns, patterns.meta, pgfplots.fillbetween, bending, shapes, arrows}
\usepackage{color}
\usepackage[hidelinks]{hyperref}
\usepackage{multirow}

\newcommand{\vertiii}[1]{{\left\vert\kern-0.1ex\left\vert\kern-0.1ex\left\vert #1 \right\vert\kern-0.1ex\right\vert\kern-0.1ex\right\vert}} 

\makeatletter
\def\hlinewd#1{%
  \noalign{\ifnum0=`}\fi\hrule \@height #1 \futurelet
   \reserved@a\@xhline}
\makeatother

\allowdisplaybreaks

\def\el{{\nonumber}}
\def\jump#1{{\llbracket {#1} \rrbracket}}

\newtheorem{thm}{Theorem}
\newtheorem{lemma}{Lemma}
\newtheorem{cor}{Corollary}
\newtheorem{rem}{Remark}
\newtheorem{hyp}{Assumption}

\usepackage[bordercolor=white,backgroundcolor=gray!30,linecolor=black,colorinlistoftodos]{todonotes}

\begin{document}

\begin{frontmatter}

\title{A posteriori error estimates for the finite element approximation \\ of the convection--diffusion--reaction equation \\ based on the variational multiscale concept}

\author[deca,cimne]{Ramon Codina\corref{ca}}
\ead{ramon.codina@upc.edu}

\author[ovgu]{Hauke Gravenkamp}
\ead{hauke.gravenkamp@ovgu.de}

\author[deca,cimne]{Sheraz Ahmed Khan}
\ead{sheraz.ahmed@upc.edu}

\address[deca]{Universitat Polit\`{e}cnica de Catalunya (UPC), Jordi Girona 1-3, Edifici C1, Barcelona, 08034, Spain}
\address[cimne]{International Centre for Numerical Methods in Engineering (CIMNE), Campus Nord UPC, Gran Capit\`a s/n,
Barcelona, 08034, Spain}
\address[ovgu]{Institute of Materials, Technologies and Mechanics, Otto von Guericke University Magdeburg, Germany}

\cortext[ca]{Corresponding author}

\begin{abstract}
In this study, we employ the variational multiscale (VMS) concept to develop a posteriori error estimates for the stationary convection-diffusion-reaction equation. The variational multiscale method is based on splitting the continuous part of the problem into a resolved scale (coarse scale) and an unresolved scale (fine scale). The unresolved scale (also known as the sub-grid scale) is modeled by choosing it proportional to the component of the residual orthogonal to the finite element space, leading to the orthogonal sub-grid scale (OSGS) method. The idea is then to use the modeled sub-grid scale as an error estimator, considering its contribution in the element interiors and on the edges. We present the results of the a priori analysis and two different strategies for the a posteriori error analysis for the OSGS method. Our proposal is to use a scaled norm of the sub-grid scales as an a posteriori error estimate in the so-called stabilized norm of the problem. This norm has control over the convective term, which is necessary for convection-dominated problems. Numerical examples show the reliable performance of the proposed error estimator compared to other error estimators belonging to the variational multiscale family.
\end{abstract}


\end{frontmatter}

\section{Introduction}

The finite element (FE) method is a numerical technique to approximate solutions of partial differential equations computationally and is widely used to solve engineering problems. In fluid mechanics, applying the standard Galerkin FE method proves to be challenging in two main ways. Firstly, the numerical approximation of convection-diffusion type problems lacks stability when convection is dominant and exhibits numerical oscillations. Secondly, it is frequently difficult to satisfy the inf-sup stability condition for mixed interpolations. Alternatively, stabilized FE methods can be used, which contain additional stabilization terms and provide proper stability without the need to satisfy the inf-sup conditions. 
Moreover, the singularly perturbed nature of the continuous problem causes the stability bound to explode when diffusion (or viscosity) approaches zero. This drawback is mitigated when stabilized methods are employed. The addition of new stabilization terms enables control over the norm of the convective term and enhances accuracy.

A comprehensive comparison of stabilization techniques applied to the convection-diffusion-reaction equation (CDRE) can be found in \cite{codina1998comparison}. The variational multiscale (VMS) method introduced by Hughes et al.~\cite{hughes1995multiscale,hughes1998variational} serves as the foundation to derive the stabilized method. The point of departure is the notion that the FE approximation cannot capture the solution precisely, hence, the VMS method is based on splitting the continuous solution into a resolved scale (coarse scale) and an unresolved scale (fine scale), where the unresolved scale, also known as the sub-grid scale (SGS), is represented or estimated analytically. The SGS model we have used is known as the orthogonal SGS (OSGS) method, introduced in \cite{codina2000stabilization}. In the OSGS approach, the SGS problem is modeled in a specific fashion by choosing the SGSs orthogonal to the FE space. The stability and convergence analysis of this method for the stationary and transient CDRE are presented in \cite{codina2011finite} and \cite{codina2002analysis}, respectively. The complete analysis for the Oseen problem, including non-uniform FE meshes, can be found in~\cite{codina2008analysis}. A detailed discussion of the VMS method and its application to computational fluid dynamics problems can be found in \cite{codina2018variational}. 
While substantial expertise has been developed in solving physical problems, it has remained a computational challenge to predict and control the error of the computed solutions. In recent years, a posteriori error estimation has made significant progress, becoming an essential tool for FE practitioners. An overview of the topic is given by Ainsworth and Oden in \cite{ainsworth1997posteriori}. See also \cite{verfurth2013posteriori} for an overview of a posteriori error estimation using the Galerkin method, including problems related to fluid mechanics. A recent survey on the development of a posteriori error estimation in fluid mechanics based on the VMS theory has been conducted by Hauke and Irisarri in~\cite{hauke2023review}. In general, a posteriori error estimation is a broad discipline that originated with Babu\v ska and Rheinbolt \cite{babuvska1978posteriori} for elliptic problems and was later successively expanded by Zienkiewicz and Zhu \cite{zienkiewicz1987simple,zienkiewicz1991adaptivity}, Eriksson and Johnson \cite{eriksson1988adaptive}, and Ainsworth and Oden \cite{ainsworth1993unified}, among others. In addition, Verf\"{u}rth investigated a posteriori error estimates for the Stokes problem \cite{verfurth1989posteriori}, for the convection-diffusion problem \cite{verfurth1998posteriori}, and for elliptic problems \cite{verfurth1994posteriori}. 
Moreover, we refer to the recent work using different numerical methods for the Navier-Stokes problem in \cite{camano2022posteriori}, for semi-linear elliptic problems in \cite{chen2023posteriori}, and for linear and non-linear singularly perturbed problems in \cite{clavero2025priori}. 
Previously, several studies have been conducted to estimate the a posteriori error for convection-diffusion-type problems. In \cite{tobiska2015robust}, an a posteriori error analysis was performed for the stationary and transient CDRE using stabilized FE schemes. In this reference, error estimates are measured in an energy-like norm for the symmetric part of the differential operator, while a dual norm is employed to bound the convective term. Likewise, Verf\"{u}rth used the same strategy to measure the error for stationary CDREs in \cite{verfurth2005robust}, for non-stationary CDREs in \cite{verfurth2005robust2}, and for non-linear non-stationary CDREs in \cite{verfurth2018posteriori}. A similar approach has been by followed by Sharma \cite{sharma2021robust} to explore a posteriori error estimates for the weak Galerkin method applied to the CDRE. Moreover, a similar choice of norm has been made by Ainsworth et al.\ in \cite{ainsworth2013fully}, where the authors derived a posteriori error estimators using the streamline-upwind Petrov-Galerkin (SUPG) method to approximate solutions of the stationary CDRE in three dimensions. Du et al.~\cite{du2021robust} also investigated recovery-type a posteriori error estimates using the SUPG method for singularly perturbed problems.

Already in the initial development of the VMS method, it was suggested by Hughes et al.~\cite{hughes1998variational} to use the SGSs as an a posteriori error estimator. One of the first attempts to estimate a posteriori the error using the VMS technology for convection-dominated problems was presented in \cite{hauke2006multiscale}, verifying that VMS-based stabilized methods are equipped with an inherent error estimator. In recent years, several studies have been conducted to derive a posteriori error estimates using the VMS method \cite{araya2005adaptive,hauke2006proper,hauke2015variational,hauke2008variational,hauke2015variationalvms}. The same strategy is applied to higher-order elements in \cite{hauke2006application,hauke2006proper} and later extended to 2D domains in \cite{hauke2008variational}, where the error contribution on the element boundaries is also discussed. In addition, error estimators have been derived for the Navier-Stokes equations in \cite{rossi2013parallel,hauke2015variationalvms,irisarri2021posteriori}, for the Stokes equations in \cite{irisarri2018posteriori}, for linear elasticity in \cite{hauke2015variational}, and for higher-order ordinary differential equations in \cite{irisarri2015variational}, all based on the same concepts.

In this study, the VMS method also serves as the foundation for the a posteriori error estimation. We propose such an estimate for the stationary CDRE using the OSGS approach. The key point is that the {\em stabilized norm} of the error (not any other norm) can be estimated by a {\em scaled} norm of the SGSs, the scaling being provided by the so-called stabilization parameters. Baiges et al.~\cite{baiges2017variational} used the same technique to estimate the error for solid mechanics problems and revealed the robust performance of the OSGS-based error estimator. The stabilized norm is precisely the norm in which stability and a priori convergence can be proved. It includes control over the convective term, which is necessary for convection-dominated problems. The same choice of norm has been made to develop the a posteriori error estimates for incompressible Navier-Stokes equations in \cite{codina2021posteriori}. 

The a posteriori error estimator (APEE) we propose is based on two components: the SGSs in the element interiors and on the element boundaries. The representation of the SGSs on the element boundaries was first introduced in \cite{codina2009subscales, codina2011finitebou}.  We provide a theoretical foundation for the proposed error estimates by performing an a posteriori error analysis of the OSGS approximation to the CDRE. We explain why this analysis is not fully satisfactory and develop two strategies that provide partial information, one following the ideas in \cite{verfurth2005robust} and the other close to the analysis in \cite{john-novo-2013}. We continue by presenting some numerical examples to demonstrate the good numerical behavior of the APEE in terms of the effectivity index and discuss its performance.

The paper is organized in the following format. In Section~2, the model problem is stated, and its functional setting is introduced along with its FE approximation, both the Galerkin method and the stabilized FE formulation following the VMS concept. Section~3 presents the main results of the a priori analysis. In Section~4, the a posteriori error analysis is conducted using two different strategies, explaining the limitations of each of them. Section~5 includes numerical examples illustrating the performance of the stabilized formulation and the APEE. Finally, the conclusions of this study are presented in Section~6.

\section{Continuous problem and finite element approximation}

\subsection{Statement of the problem}

Let $\Omega \subset \mathbb{R}^d$ be an open bounded domain, where $d$ represents the number of spatial dimensions. The boundary of the domain is denoted by $\Gamma$. The boundary value problem we consider consists of finding $u:\Omega\rightarrow \mathbb{R}$ such that
\begin{align} \label{modprob}
  \mathcal{L}u :=      -k \Delta u + a \cdot \nabla u + su &= f \qquad \text{in} \quad  \Omega \\
        u &= 0 \qquad  \text{on} \quad \Gamma
\end{align}
where $k > 0$ is the diffusion coefficient, $a \in \mathbb{R}^d$ is the advection velocity, $s \geq 0$ is the reaction coefficient, and, on the right-hand side, $f$ is the forcing term. 

Our objective in this work is to present an a posteriori analysis {\em in the simplest possible setting}, trying to emphasize the critical points without distracting ourselves with technicalities that, despite being relevant, can be found elsewhere. In particular, we will consider $k$, $a$ and $s$ to be constants and limit the discussion to homogeneous Dirichlet conditions. Furthermore, $f$ will be assumed to be an FE function (see below).

Let us introduce some notation for the functional setting of the problem. For a region $\omega\subset \mathbb{R}^d$, we denote by $L^2(\omega)$ the space of square integrable functions in $\omega$ and by $H^m(\omega)$ the space of functions with derivatives of order up to $m \in \mathbb{N}$ in $L^2(\omega)$, with $H^1_0(\omega)$ the space of functions in $H^1(\omega)$ vanishing on $\partial\omega$, its dual being $H^{-1}(\omega)$. The inner product in $L^2(\omega)$ is denoted as $(\cdot,\cdot)_\omega$, and the norm as $\Vert \cdot \Vert_\omega$. The duality pairing based on the integral is denoted by $\langle\cdot,\cdot\rangle_\omega$. In all cases, the subscript is omitted when $\omega= \Omega$. For any other Banach space $X$, its norm is written as $\Vert \cdot \Vert_X$, and if $X$ is endowed with a semi-norm we denote it by $\vert \cdot\vert_X$, the exception being $\Vert \cdot \Vert_{H^m(\Omega)} \equiv \Vert \cdot \Vert_{H^m}$, $m\in \mathbb{Z}$. 

Setting $V_c= H^{1}_{0}(\Omega)$, the variational form of the problem can be written as: find $u\in V_{c}$ such that
\begin{equation} \label{variform}
	B(u,v)=L(v)\ \ \ \ \ \ \forall v\in V_{c} 
\end{equation}
with
\begin{align*}
    B(u,v)&=k (\nabla u, \nabla v)+(a \cdot \nabla u, v)+s(u,v)\\
	L(v)&=\left\langle f,v\right\rangle
\end{align*}
where $B$ and $L$ represent the bilinear and linear forms of the problem, respectively; $B$ is defined on $H^{1}_{0}(\Omega) \times H^{1}_{0}(\Omega)$ and $L$ on $H^{1}_0(\Omega)$.

\subsection{Galerkin finite element approximation}

Let ${\cal T}_h = \{K\}$ be an FE partition of $\Omega$. This mesh will be considered throughout as quasi-uniform with diameter $h$.
In the following, we will consider a space smaller than $V_{\rm c}$ associated to ${\cal T}_h$, defined as
$$V = \{ v\in H^1_0(\Omega) ~\vert~ v\vert_K \in H^2(K) ~ \forall K\in {\cal T}_h\}$$
The collection of interior edges of ${\cal T}_h$ is denoted as ${\cal E}_h = \{ E\}$. Since we consider only homogeneous boundary conditions, we refrain from including edges on $\Gamma$. We shall write 
\begin{align*}
\Vert \cdot \Vert^m_{{\cal T}_h} := \sum_K \Vert \cdot \Vert^m_K, \qquad \Vert \cdot \Vert^m_{{\cal E}_h} := \sum_E \Vert \cdot \Vert^m_E, \quad m = 1, 2
\end{align*}
and an analogous notation is used when norms are replaced by semi-norms. Summation extends over all $K\in {\cal T}_h$ or over all $E \in {\cal E}_h$.

Let $V_h\subset V$ be the approximating FE space for the trial and test functions constructed from ${\cal T}_h$. The variational problem can now be approximated directly using the Galerkin approximation. We construct the finite space $V_h$ as
\begin{equation*}
    V_h=\left\{ v_h \in H^{1}_{0}(\Omega) ~|~ v_h\vert_K  \in \mathbb{P}_{p}(K), p \geq 1, ~K \in {\cal T}_h \right\}
\end{equation*}
where $\mathbb{P}_p(K)$ denotes the set of complete polynomials of degree $p$ in $K\in {\cal T}_h$. Then, the FE approximation of the variational problem \eqref{variform} consists of finding $u_h \in V_{h}$ such that
\begin{equation} \label{finapprox}
	B(u_h,v_h)=L(v_h)\quad \forall v_h\in V_{h} 
\end{equation}

\subsection{Stabilized finite element formulation}

The formulation we shall analyze is based on the VMS concept, which relies on the splitting $V = V_h\oplus V'$. The distinctive feature of our approach is that we take the space of SGSs $V'$ as $L^2$-orthogonal to the finite element space $V_h$, i.e., $V' = V_h^\bot$, leading to the OSGS method. We do not provide a detailed motivation for the formulation here, as it is discussed in \cite{codina2000stabilization,codina2002stabilized,codina2018variational} regarding the approximation of the SGSs in the element interiors and in \cite{codina2009subscales} for the approximation of the SGSs on the element boundaries. Here, we directly state the final version of the method.

We consider $u = u_h + u'$, with $u'_K$ an approximation to $u'$ in the interior of the elements and $u'_E$ on the edges. These approximations are given by:
\begin{alignat}{3}
& u'_K = \tau_K P_h^\bot (R_K) && \hbox{SGSs in the element interiors}  \label{eq:SGSs1}\\
& u'_E = \tau_E   R_E && \hbox{SGSs on the element edges}  \label{eq:SGSs2}\\
& R_K := [ f - (-k \Delta u_h + a\cdot\nabla u_h + s u_h)]\vert_K\quad  && \hbox{Residual in the element interiors}  \nonumber \\
& R_E :=  k \jump{\partial_n u_h}\vert_E && \hbox{Residual on the element edges}  \nonumber \\
& \tau_K = \left( c_1 \frac{k}{h^2} +  c_2  \frac{\vert a\vert  }{h}  + c_3 s\right)^{-1}&& \hbox{Stabilization parameter in the element interiors} \label{eq:tauK} \\
& \tau_E = c_4  \frac{1}{h}  \tau_K && \hbox{Stabilization parameter on the element edges} \label{eq:tauE}
\end{alignat}
In these expressions, $c_1$, $c_2$, $c_3$, $c_4$ are algorithmic constants that depend on the order of the polynomial $p$ of the finite element approximation, which will be considered constant. These constants are related to the inverse estimate \eqref{eq:invest} stated later, which determines how they have to be chosen. In the numerical examples, we take these constants as $c_1 = 4$, $c_2 = 2$, $c_3 = 1$, and $c_4 = 1/3$ for linear elements. The symbol $\jump{\cdot}$ denotes the jump along the normal, and $\partial_n u = n \cdot \nabla u$ is the derivative in the direction of the normal $n$ to $E$, exterior to an element $K$ such that $E\subset\partial K$.

The OSGS formulation we consider is: find $u_h\in V_h$ such that
\begin{align}
& B_{\rm stab}(u_h,v_h) := B(u_h,v_h) +S(u_h,v_h)  =  \langle f , v_h\rangle \qquad  \forall v_h \in V_h \label{eq:osgs}
\end{align}
where the stabilization terms are
\begin{align}
S(u,v) :=  \sum_K \tau_K \langle k \Delta v + a\cdot\nabla v - s v ,  P_h^\bot ( -k \Delta u + a\cdot\nabla u + s u )\rangle_K 
- \sum_E \tau_E  \left\langle k\jump{\partial_n u} , k\jump{\partial_n v}\right\rangle_E \label{eq:stab-term}
\end{align}
with $P_h^\bot  = I - P_h$, and $P_h$ is the projection defined by 
\begin{align*}
\sum_K \langle v_h , P_h w \rangle_K  = \sum_K \langle v_h , w\rangle_K \quad \forall v_h \in V_h, ~\hbox{with}~w\vert_K \in L^2(K) ~\forall K\in {\cal T}_h
\end{align*}
Again, for the sake of simplicity, we assume in this paper that $f$ is an FE function, and therefore $P_h^\bot(f) = 0$. Otherwise, in the a posteriori error estimates to be obtained, there would be a contribution stemming from the FE approximation of $f$. The same applies if non-homogeneous Dirichlet conditions or Neumann conditions are applied with data that do not belong to the appropriate FE space. Note that the first term in \eqref{eq:stab-term} can be written as
\begin{align}
\sum_K \tau_K \langle k \Delta v + a\cdot\nabla v - s v ,  P_h^\bot ( -k \Delta u + a\cdot\nabla u + s u )\rangle_K 
= \sum_K \tau_K \langle  - {\cal L}^* v , P_h^\bot ( {\cal L} u - f) \rangle_K
= \sum_K  \langle   {\cal L}^* v , u'_K \rangle_K \label{eq:SGSsf}
\end{align}
where ${\cal L}^* v = -k \Delta v - a\cdot\nabla v + s v$, i.e., ${\cal L}^*$ is the adjoint of ${\cal L}$.

\begin{rem} When evaluated with FE functions, the terms $-sv$ and $su$ in $S(u,v)$ do not need to be included, as their projection orthogonal to the FE space is zero. However, they are needed for consistency, i.e., to guarantee that $B_{\rm stab}(u,v) = L(v)$ when $u$ is the solution of the continuous problem.
\end{rem}

\begin{rem} The sign of the boundary term in $S(u,v)$ is negative and can therefore be deleted, since it deteriorates stability rather than contributing to it, and consistency is ensured without this term. However, a similar term appears, for example, in the Stokes problem, where the boundary term is necessary to stabilize discontinuous pressures~\cite{codina2009subscales}.
\end{rem}

In the numerical examples, we will also consider the possibility of dropping $P_h^\bot$ in expression~\eqref{eq:SGSsf}, leading to a classical residual-based stabilized method that we will call Algebraic SGS (ASGS) formulation~\cite{codina2018variational}.

\section{A priori analysis}

\subsection{Notation and preliminaries}

We shall make frequent use of the inverse inequality
\begin{align}
\Vert \nabla v_h\Vert_K \leq {C_{\rm inv}} \frac{1}{h} \Vert v_h\Vert_K\label{eq:invest}
\end{align}
as well as the trace inequality in the form (see \cite{veeser-verfurth-2009}, Corollary 4.5):
\begin{align}
\Vert v \Vert^2_{\partial K} \leq {C_{\rm trace}} \left( \frac{1}{h} \Vert v\Vert^2_K +  \Vert v \Vert_K \Vert \nabla v \Vert_K\right) \label{eq:tracein}
\end{align}
for $v\in V$, $C_{\rm inv}$ and $C_{\rm trace}$ being positive constants. Using Young's inequality for the second term, for any $\alpha > 0$ we can rewrite this as 
\begin{align}
\Vert v \Vert^2_{\partial K} \leq {C_{\rm trace}} \left[ \left( 1 + \frac{\alpha}{2}\right) \frac{1}{h} \Vert v\Vert^2_K + \frac{1}{2\alpha} {h}\Vert \nabla v \Vert^2_K\right]\label{eq:tracein1}
\end{align}
In what follows, the symbol $\lesssim$ stands for $\leq$ up to positive constants, and likewise $\gtrsim$ stands for $\geq$ up to positive constants. We shall write $\sim$ when both $\lesssim$ and $\gtrsim$ hold, i.e., for quantities with the same asymptotic behavior. In our analysis, we will not track these constants, but in all cases, we make sure that they are dimensionless and independent of the discretization and the data, including the physical parameters ($k$, $a$, and $s$). 

For piecewise polynomials $v_h$, the last term in \eqref{eq:tracein1} can be dropped using an inverse estimate, and we get:
\begin{align}
\Vert v_h \Vert^2_{\partial K} \lesssim \frac{1}{h} \Vert v_h\Vert^2_K \label{eq:tracein2}
\end{align}

For smooth enough functions $u$, we will make use of the interpolation estimate
\begin{align}
\inf_{\tilde{u}_h\in V_h} \Vert u - \tilde{u}_h\Vert_{H^s} \lesssim {h^{p+1-s}}\Vert u \Vert_{H^{p+1}}\label{eq:intest}
\end{align}

We denote by $\tilde{K}$ the union of the elements in contact with $K$ and by $\tilde{E}$ the union of the elements that share $E$. Thus, if $K_1$ and $K_2$ are two neighboring elements, $K_1\cap K_2 = E$, $K_1 \cup K_2 = \tilde{E}$. For any piecewise polynomial function $v_h$ vanishing on $\partial{\tilde{E}}$ and taking the maximum on $E$, we have that
\begin{align}
\Vert v_h \Vert^2_{\tilde{E}} \lesssim {h} \Vert v_h \Vert^2_{E} \label{eq:maxE}
\end{align}

Let us observe that, for $a,b > 0$:
\begin{align}
&\frac{1}{a+b} \leq \min\left\lbrace \frac{1}{a} ,  \frac{1}{b}\right\rbrace \leq {2} \frac{1}{a+b}\label{eq:minab}
\end{align}
This trivial observation is important from the implementation point of view, since it is always convenient to avoid the ``ifs" involved in taking the minimum. 

Let us define the classical stabilized norm:
\begin{align}
\vertiii{v}^2 := k \Vert \nabla v\Vert^2 + s\Vert v\Vert^2 + \tau_K \Vert  a\cdot\nabla v \Vert^2  \label{eq:stab-norm}
\end{align}
This norm can be localized to element domains in the natural way:
\begin{align*}
\vertiii{v}^2_K := k \Vert \nabla v\Vert^2_K + s\Vert v\Vert^2_K + \tau_K \Vert  a\cdot\nabla v \Vert_K^2  
\end{align*}

Finally, let us also introduce the interpolation error $i_u := u - P_h u = P_h^\bot u$. 

\subsection{A quasi inf-sup stability result}

Let us start with a technical result. It turns out that the normal derivative on the element faces of a function $u\in V$ can be bounded as follows:

\begin{lemma} \label{du-on-dK}
For any $u\in V$ there holds
\begin{align}
\Vert \partial_n u \Vert_{\partial K} \lesssim \frac{1}{h^{1/2}} \Vert \nabla u \Vert_K + {h^{1/2}} \vert i_u  \vert_{H^2(K)}, \quad K \in {\cal T}_h\label{eq:du-on-dK}
\end{align}
\end{lemma}
\begin{proof}
From the trace inequality \eqref{eq:tracein1} with $\alpha = 1$ we have that
\begin{align}
\Vert \partial_n u \Vert_{\partial K}
& \lesssim \frac{1}{h^{1/2}} \Vert \nabla u \Vert_K  + {h^{1/2}} \Vert \nabla  \nabla u \Vert_K \nonumber \\
& \lesssim \frac{1}{h^{1/2}} \Vert \nabla u \Vert_K + {h^{1/2}} \Vert \nabla \nabla i_u  \Vert_K + h^{1/2} \Vert \nabla \nabla P_h u \Vert_K \nonumber\\
& \lesssim \frac{1}{h^{1/2}} \Vert \nabla u \Vert_K + {h^{1/2}} \Vert \nabla \nabla i_u  \Vert_K +  \frac{1}{h^{1/2}}  \Vert \nabla P_h u \Vert_K\nonumber
\end{align}
The inverse inequality has been used in the last step. The result follows from the $H^1$-stability of $P_h$.
\end{proof}

As it will be discussed in the following section, the main difficulty in the a posteriori error analysis is that $B_{\rm stab}$ is neither continuous nor stable in the stabilized norm \eqref{eq:stab-norm} {\em in the whole space $V$}. Let us concentrate now on stability. In \cite{badia-codina-2006-1}, the concept of $\Lambda$-coercivity was introduced, which immediately leads to inf-sup stability. We will show now that we can prove $\Lambda$-coercivity for $B_{\rm stab}$ (and therefore inf-sup stability) {\em up to the interpolation error}.

\begin{thm}[Quasi inf-sup stability for $B_{\rm stab}$ in $V$]\label{thm1} Suppose that the algorithmic constants $c_1$ and $c_2$ are sufficiently large, $c_3 \leq 1$ and $c_4$ is sufficiently small. Then, for any $u\in V$ there exists $v = \Lambda (u) = u + \tilde{u}_h$, with $\tilde{u}_h\in V_h$, such that 
\begin{align*}
B_{\rm stab}(u,\Lambda(u)) \gtrsim  \vertiii{u}^2  - \zeta^2 (i_u )
\end{align*}
where 
\begin{align*}
\zeta^2 (i_u ) :=  \sum_K \tau_K k^2 \vert  i_u   \vert^2_{H^2(K)} + \sum_E \tau_E k^2 \Vert \jump{\partial_n i_u }\Vert_E^2
\end{align*}
\end{thm}

\begin{proof} We have that
\begin{align}
B_{\rm stab}(u,u) 
& \gtrsim k\Vert \nabla u\Vert^2 +  s \Vert u \Vert^2 + \tau_K \Vert P_h^\bot (a\cdot \nabla u) \Vert^2 \nonumber\\
&  - \sum_K \tau_K k^2 \Vert P_h^\bot (\Delta u)\Vert^2_K  - \sum_K \tau_K s^2 \Vert P_h^\bot (u)\Vert^2_K -  \sum_E \tau_E \Vert \jump{k \partial_n u}\Vert^2_E\label{eq:ee1}
\end{align}
We cannot use an inverse estimate to control the Laplacian term. However, using the $H^1$-stability of $P_h$ we have that
\begin{align}
\tau_K k^2 \Vert P_h^\bot (\Delta u)\Vert^2_K & \lesssim \tau_K k^2 \Vert P_h^\bot (\Delta i_u)\Vert^2_K +  \tau_K k^2 \Vert P_h^\bot (\Delta P_h u )\Vert^2_K\nonumber\\
&  \lesssim \tau_K k^2 \Vert P_h^\bot (\Delta i_u)\Vert^2_K + \tau_K k^2 \frac{1}{h^2} C^2_{\rm inv}\Vert \nabla (P_h u ) \Vert^2_K\nonumber\\
& \lesssim \tau_K k^2 \Vert P_h^\bot (\Delta i_u)\Vert^2_K + \frac{1}{c_1} C^2_{\rm inv} k  \Vert \nabla u\Vert_K^2\el
\end{align}
Thus, for $c_1$ sufficiently large, the second term can be controlled by the first term in the RHS of \eqref{eq:ee1}. For the last term in \eqref{eq:ee1} we can proceed similarly:
\begin{align}
\tau_E \Vert \jump{k \partial_n u}\Vert^2_E 
& \lesssim \tau_E k^2 \left( \Vert \jump{\partial_n i_u}\Vert^2_E + \Vert \jump{\partial_n P_h u}\Vert^2_E\right)\nonumber\\
& \lesssim  \tau_E k^2 \Vert \jump{\partial_n i_u}\Vert^2_E + c_4 \frac{1}{h} \frac{h^2}{c_1 k} k^2 C^2_{\rm trace} \frac{1}{h} \Vert \nabla u\Vert^2_{\tilde{E}} \nonumber\\
& \lesssim  \tau_E k^2 \Vert \jump{\partial_n i_u}\Vert^2_E + \frac{c_4}{c_1} C^2_{\rm trace}  k \Vert \nabla u\Vert^2_{\tilde{E}}\el
\end{align}
Again, for $c_1$ sufficiently large or $c_4$ sufficiently small, the second term can be controlled by the first term in the RHS of \eqref{eq:ee1}. Finally, the 5th term in the RHS of  \eqref{eq:ee1} can be absorbed by the 2nd, simply requiring $c_3 \leq 1$. Thus, we have that
\begin{align}
B_{\rm stab}(u,u) 
& \gtrsim k\Vert \nabla u\Vert^2 +  s \Vert u \Vert^2 
+ \tau_K \Vert P_h^\bot (a\cdot \nabla u) \Vert^2  - \zeta^2 (i_u)\label{eq:ee2}
\end{align}

Let us now consider $v^0_h = \tau_K P_h(a\cdot \nabla u)$, which is an FE function to which inverse estimates can be applied. We have that
\begin{align*}
B_{\rm stab}(u,v^0_h) & \geq - k \Vert \nabla u \Vert \tau_K \Vert \nabla P_h(a\cdot\nabla u)\Vert + \tau_K \Vert P_h(a\cdot\nabla u_h)\Vert^2
-  s \Vert u \Vert \tau_K \Vert P_h(a\cdot\nabla u) \Vert \\
& - \sum_K \tau_K^2 \Vert k \Delta P_h(a\cdot\nabla u) + a\cdot\nabla P_h(a\cdot\nabla u) \Vert_K \Vert P_h^\bot(- k \Delta u + a\cdot\nabla u + su)\Vert_K \\
& - \sum_E \tau_E \tau_K \Vert k\jump{\partial_n  u} \Vert_E \Vert k\jump{\partial_n P_h(a\cdot\nabla u)} \Vert_E
\end{align*} 
Let us bound the different terms, starting with the first one:
\begin{align}
& k \Vert \nabla u \Vert \tau_K \Vert \nabla P_h(a\cdot\nabla u)\Vert \leq k \tau_K  C_{\rm inv} \frac{1}{h} \Vert \nabla u \Vert \Vert  P_h(a\cdot\nabla u)\Vert \nonumber\\
& \qquad \leq k^{1/2} \tau^{1/2}_K  \Vert \nabla u \Vert \Vert  P_h(a\cdot\nabla u)\Vert \nonumber\\
& \qquad  \leq \frac{1}{2\alpha_1} k  \Vert \nabla u \Vert^2 + \frac{\alpha_1}{2} \tau_K \Vert {P_h(a\cdot\nabla u )} \Vert^2 \nonumber\\
& s \Vert u \Vert \tau_K \Vert P_h(a\cdot\nabla u) \Vert \leq s^{1/2} \tau_K^{1/2}  \Vert u \Vert   \Vert P_h(a\cdot\nabla u) \Vert \nonumber\\
& \qquad \leq \frac{1}{2\alpha_2} s  \Vert  u \Vert^2 + \frac{\alpha_2}{2} \tau_K \Vert { P_h(a\cdot\nabla u)} \Vert^2\nonumber\\
& \sum_K \tau_K^2 \Vert k \Delta P_h(a\cdot\nabla u) + a\cdot\nabla P_h(a\cdot\nabla u) \Vert_K \Vert P_h^\bot(- k \Delta u + a\cdot\nabla u + su)\Vert_K\nonumber\\
& \qquad \leq \sum_K \tau_K^2 \left( k C^2_{\rm inv}\frac{1}{h^2} + C_{\rm inv}\frac{\vert a\vert }{h}\right)\Vert P_h(  a\cdot\nabla u) \Vert_K \nonumber\\
& \qquad \quad \times \sum_K \left( k\Vert P_h^\bot( \Delta i_u) \Vert_K + k \Vert P_h^\bot( \Delta (P_h u)) \Vert_K +  \Vert P_h^\bot(a\cdot\nabla u)\Vert_K 
+ s \Vert P_h^\bot u\Vert_K \right) \nonumber\\
& \qquad \lesssim \tau^{1/2}_K \Vert  P_h(a\cdot\nabla u) \Vert   \left( \zeta(i_u) + \tau^{1/2}_K C_{\rm inv} \frac{1}{h} k\Vert \nabla u \Vert + \tau_K^{1/2} s \Vert u \Vert \right)+ \beta \tau_K \Vert a\cdot\nabla u\Vert^2 \label{eq:conv-11}\\
& \qquad \lesssim \tau^{1/2}_K  \Vert  a\cdot\nabla u \Vert \left( \zeta(i_u) + k^{1/2}\Vert \nabla u \Vert + s^{1/2}\Vert u \Vert\right) 
+ \beta \tau_K \Vert   a\cdot\nabla u\Vert^2 \nonumber\\
& \qquad \lesssim \frac{1}{\alpha_3} k  \Vert \nabla u \Vert^2 + \frac{1}{\alpha_3} \zeta^2(i_u) + \frac{1}{\alpha_3} s  \Vert u \Vert^2+ \left({\alpha_3} 
+ \beta\right) \tau_K \Vert {  a\cdot\nabla u}  \Vert^2 \el
\end{align} 
We have assumed that $c_2 > C_{\rm inv}$, and, thus, the constant $\beta$ is $\beta < 1$. For the last term we have that:
\begin{align*}
& \sum_E \tau_E \tau_K\Vert k\jump{\partial_n  u} \Vert_E \Vert k\jump{\partial_n P_h(a\cdot\nabla u)} \Vert_E 
\lesssim \sum_K \tau_{E} \tau_K k^2 \Vert {\partial_n u}  \Vert_{\partial K}\Vert {\partial_n P_h(a\cdot\nabla u)} \Vert_{\partial K}
\end{align*} 
From Lemma~\ref{du-on-dK} we obtain
\begin{align*}
& \sum_E \tau_E \tau_K\Vert k\jump{\partial_n  u} \Vert_E \Vert k\jump{\partial_n P_h(a\cdot\nabla u)} \Vert_E \nonumber\\
& \qquad \lesssim  \sum_K \tau_{E} \tau_K k^2  \left( \frac{1}{h^{1/2}} \Vert {\nabla u} \Vert_{K} + {h^{1/2}} \vert i_u \vert_{H^2(K)} 
\right) \frac{1}{h^{1/2}}\Vert {\nabla P_h(a\cdot\nabla u)} \Vert_{K} \\
& \qquad \lesssim  \sum_K \tau_{E} \tau_K k^2  \frac{1}{h} \left( \frac{1}{h} \Vert {\nabla u} \Vert_{K} 
+  \vert i_u \vert_{H^2(K)} \right)\Vert P_h(a\cdot\nabla u) \Vert_{K} \\
& \qquad \lesssim  \sum_K \frac{1}{h}\tau^2_K k^2  \frac{1}{h} \left( \frac{1}{h} \Vert {\nabla u} \Vert_{K} 
+  \vert i_u \vert_{H^2(K)} \right)\Vert P_h(a\cdot\nabla u) \Vert_{K} \\
& \qquad \lesssim  \sum_K  \tau_K   k \left( \frac{1}{h} \Vert {\nabla u} \Vert_{K} +  \vert i_u \vert_{H^2(K)} \right) \Vert P_h(a\cdot\nabla u) \Vert_{K}\\
& \qquad \lesssim  \sum_K \left( k^{1/2}   \Vert {\nabla u} \Vert_{K} +  \tau_K^{1/2}   k \vert i_u \vert_{H^2(K)}\right) \tau_K^{1/2} \Vert P_h(a\cdot\nabla u) \Vert_{K} \\
& \qquad \lesssim \frac{1}{\alpha_4}  k \Vert {\nabla u} \Vert^2 +\frac{1}{\alpha_4}\zeta^2(i_u)  + {\alpha_4} \tau_K \Vert P_h(a\cdot\nabla u) \Vert^2
\end{align*} 
Taking $\alpha_i$, $i = 1,2,3,4$ sufficiently small, such that
\begin{align*}
{\alpha_1} + {\alpha_2} + {\alpha_3} + {\alpha_4} + \beta < 1
\end{align*} 
it follows that
\begin{align}
B_{\rm stab}(u,v^0_h) \gtrsim \tau_K \Vert { P_h(a\cdot\nabla u)} \Vert^2 - \tau_K \Vert { P_h^\bot(a\cdot\nabla u)} \Vert^2 - k \Vert \nabla u\Vert^2 - s\Vert u\Vert^2  - \zeta^2 (i_u)\label{eq:stab2}
\end{align}
From (\ref{eq:ee2}) and (\ref{eq:stab2}), taking $\Lambda(u) = u + \gamma v_h^0$, with $\gamma$ sufficiently small, the theorem follows.\end{proof}

\begin{cor}[inf-sup stability for $B_{\rm stab}$ in $V_h$]
For all $u_h\in V_h$ there exists $v_h\in V_h$ such that
\begin{align*}
B_{\rm stab}(u_h,v_h) \gtrsim \vertiii{u_h}\,\vertiii{v_h}
\end{align*}
\end{cor}

\begin{proof}
When $u_h\in V_h$, the interpolation error is zero, and therefore $\zeta(i_{u_h}) = 0$. Setting $v_h = \Lambda(u_h)$ in the previous theorem, we have that
\begin{align*}
B_{\rm stab}(u_h,v_h) \gtrsim \vertiii{u_h}^2 
\end{align*}
On the other hand, it is immediately checked that
\begin{align}
\vertiii{v_h^0}^2  =  \vertiii {\tau_K P_h(a\cdot\nabla u_h)}^2 & = 
k \tau_K^2 \Vert \nabla P_h(a\cdot\nabla u_h)\Vert^2 
+ s \tau_K^2  \Vert P_h(a\cdot\nabla u_h) \Vert^2 \nonumber\\
& + \tau_K^3 \Vert  a\cdot\nabla P_h(a\cdot\nabla u_h) \Vert^2  +  \tau_E\tau_K^2 \Vert k \jump{\partial_n P_h(a\cdot\nabla u_h)} \Vert^2_{{\cal E}_h} \nonumber\\
& \lesssim \tau_K \Vert a\cdot\nabla u_h \Vert^2 \leq \vertiii{u_h}^2 \label{eq:conv-2}
\end{align} 
and thus $\vertiii{v_h} = \vertiii{u_h + \gamma v^0_h}\lesssim\vertiii{u_h}$. 
\end{proof}

\subsection{Convergence}

Even though it is not used in the a posteriori error analysis, for completeness, we also include an a priori error estimate without proof (see \cite{Codina2018hp}):

\begin{thm}[Convergence] Under the assumptions of Theorem~1, for smooth enough solutions of the continuous problem $u$ there holds:
\begin{align}
\vertiii{u-u_h} \lesssim E(h) :=  \left[ k^{1/2}  
+ \min\Bigl\{ \tau_K^{-1/2} , \frac{\vert a\vert}{k^{1/2}} \Bigr\}  {h}
+ s^{1/2} {h}
+ \tau_K^{1/2}\vert a \vert \right]
 {h^p}\Vert u \Vert_{H^{p+1}} 
 \end{align} 
\end{thm}

\section{A posteriori analysis}

In view of the fact that the a priori analysis yields stability and convergence in the stabilized norm $\vertiii{\cdot}$, it would be desirable to have an a posteriori error bound also in this norm. However, the classical a posteriori analysis requires stability and boundedness in the whole space $V$ of the bilinear form of the problem, and these properties {\em do not hold} in our case. To explain the difficulty, consider the problem: find $u\in V$ such that
$B_W (u,v) = L(v)$ for all $v\in V$, where $B_W = B$ for the Galerkin method and $B_W = B_{\rm stab}$ for the OSGS method. If we could find norms $\Vert\cdot \Vert_{W_1}$ and $\Vert\cdot \Vert_{W_2}$ such that
\begin{alignat*}{3}
& \hbox{For all $u,v\in V$}\quad  B_W(u,v) \lesssim \Vert u \Vert_{W_1} \Vert v \Vert_{W_2} && \qquad \hbox{Continuity} \\
& \hbox{For all $u\in V$ there exists $v\in V$ such that} \quad  B_W(u,v) \gtrsim \Vert u \Vert_{W_1} \Vert v \Vert_{W_2}  &&\qquad \hbox{inf-sup stability} 
\end{alignat*}
we could perform a more or less standard analysis to obtain an a posteriori estimate for $\Vert u-u_h \Vert_{W_1}$. This analysis would be agnostic of whether $u_h$ has been obtained from the Galerkin method or from the OSGS formulation. However, the bilinear form of the OSGS method is neither continuous nor inf-sup stable in the stabilized norm in the whole space~$V$.

In the following, we will adapt to the OSGS formulation two ideas that can be found in the literature. The first is due to Verf\"urth~\cite{verfurth2005robust}. It consists of taking $B_W = B$ and finding appropriate norms $\Vert\cdot \Vert_{W_1}$ and $\Vert\cdot \Vert_{W_2}$ in which continuity and inf-sup stability holds. The second approach is due to John and Novo~\cite{john-novo-2013}. In this case, $B_W = B_{\rm stab}$ and $\Vert\cdot \Vert_{W_1} = \Vert\cdot \Vert_{W_2} = \vertiii{\cdot}$. However, in this case, neither continuity nor inf-sup stability holds exactly, but up to an interpolation error (see Theorem~\ref{thm1} and Lemma~\ref{lem:for-thm40} below). Apart from adapting the analysis in \cite{verfurth2005robust} and \cite{john-novo-2013} to the OSGS formulation, some of our proofs differ significantly from those in these references.

\subsection{Verf\"urth's approach}\label{sec:verf}

Let us introduce the following norms:
\begin{align}
& \Vert v \Vert^2_G  = k \Vert \nabla v \Vert^2 +  s \Vert v\Vert^2, \qquad \Vert v \Vert^2_{G,\omega} = k \Vert \nabla v \Vert^2_\omega +  s \Vert v\Vert^2_\omega \nonumber\\
& \Vert v \Vert_{G^\ast}  = \sup_{w\in V, v\not = 0} \frac{\langle v , w\rangle}{\Vert w \Vert_G} \nonumber\\
& \Vert v \Vert_S  = \Vert v \Vert_G + \Vert a\cdot\nabla v \Vert _{G^\ast}\el
\end{align}

\begin{lemma}[boundedness of $B$] \label{lem:bB}
There holds:
\begin{align}
& \hbox{For all $u,v\in V$} \quad  B(u,v) \lesssim \Vert u \Vert_{S} \Vert v \Vert_{G} \label{eq:upperb} \\
& \hbox{For all $u\in V$ there exists $v\in V$ such that} \quad  B(u,v) \gtrsim \Vert u \Vert_{S} \Vert v \Vert_{G} \label{eq:lowerb}
\end{align}
\end{lemma}

\begin{proof}
Proving \eqref{eq:upperb} is trivial, simply noting that
$$ (a \cdot\nabla u , v ) \leq \Vert a\cdot\nabla u \Vert _{G^\ast} \Vert v \Vert_G$$
For  \eqref{eq:lowerb}, let $u\in V$ be given, and let us pick $w \in V$ such that
$$ (a \cdot\nabla u , w ) \geq \frac{1}{2} \Vert a\cdot\nabla u \Vert _{G^\ast} \Vert w \Vert_G$$
Set now $v = 2 u + {\Vert u \Vert_G} \Vert w \Vert_G^{-1} w$. Then the following holds:
\begin{align}
B(u,v) & = 2 B(u,u) + {\Vert u \Vert_G} \Vert w \Vert_G^{-1} B(u,w) \nonumber\\
& \geq 2 \Vert u \Vert_G^2 +  {\Vert u \Vert_G} \Vert w \Vert_G^{-1} \left[  \frac{1}{2} \Vert a\cdot\nabla u \Vert _{G^\ast} \Vert w \Vert_G 
- {\Vert u \Vert_G} \Vert w \Vert_G \right] \nonumber\\
& \geq \Vert u \Vert_G^2 +  \frac{1}{2}  {\Vert u \Vert_G} \Vert a\cdot\nabla u \Vert _{G^\ast} \nonumber\\
& \geq \frac{1}{2} \Vert u \Vert_S \Vert u \Vert_G\el
\end{align}
The result follows noting that $\Vert v \Vert_G \leq  3 \Vert u \Vert_G$.
\end{proof}

\begin{lemma}[Upper and lower bounds for the solution]\label{lem:ulb}
Let $u \in V$ satisfy
$ B(u,v) = L(v)$ for all  $v \in V$. 
Then
\begin{align}
\Vert L \Vert_{G^\ast } \lesssim \Vert u \Vert_S \lesssim \Vert L \Vert_{G^\ast } 
\end{align}
\end{lemma}

\begin{proof}
From \eqref{eq:upperb}: for all $v\in V$
$$ \Vert u \Vert_S \Vert v \Vert_G \gtrsim B(u,v) = L(v) \Rightarrow  \Vert u \Vert_S \gtrsim \frac{L(v)}{\Vert v \Vert_G}$$
Taking the sup yields $\Vert L \Vert_{G^\ast } \lesssim \Vert u \Vert_S$. Now, using \eqref{eq:lowerb}, given $u$ there exists $v\in V$ such that
$$ \Vert u \Vert_S \Vert v \Vert_G \lesssim B(u,v) = L(v) \lesssim \Vert L \Vert_{G^\ast } \Vert v \Vert_G$$ from where
$\Vert L \Vert_{G^\ast } \gtrsim \Vert u \Vert_S$ follows. 
\end{proof}

Given $w\in V$, let $R_w: V \longrightarrow \mathbb{R}$ be the residual form, defined as 
$$R_w (v) = L(v) - B(w, v) \quad \forall v \in V$$

\begin{lemma}[Expression of $R_{u_h}$ in terms of residuals]\label{lem3}
There holds:
\begin{align}
& R_{u_h}(v) = \sum_K \langle R_K , v\rangle_K + \sum_E \langle R_E, v\rangle_E  \quad \forall v\in V\label{eq:res1} \\
& R_{u_h}(v-P_h v) = \sum_K \langle P_h^\bot (R_K) , v-P_h v\rangle_K + \sum_E \langle R_E, v-P_h v\rangle_E  \quad \forall v\in V\label{eq:res2}
\end{align}
\end{lemma}

\begin{proof}
Eq.~\eqref{eq:res1} follows directly from integration by parts, whereas Eq.~\eqref{eq:res2} is obtained noting that
$$\sum_K \langle R_K , v-P_h v\rangle_K = \sum_K \langle R_K , P_h^\bot v \rangle_K  =  \sum_K \langle P_h^\bot (R_K) , P_h^\bot v \rangle_K$$
for any $v\in V$.
\end{proof}

This result, though trivial, highlights how natural it is to introduce the orthogonal projection applied to the residual. In the case of stabilized formulations, it justifies the use of OSGS rather than classical residual-based methods such as ASGS.

\begin{lemma}[Bounding the error in terms of the residual] \label{lem:error-residual}
Let $e = u - u_h$, with $u$ the solution of the continuous problem and $u_h\in V_h$ given. Then
\begin{align}
\Vert R_{u_h} \Vert_{G^\ast } \lesssim \Vert e \Vert_S \lesssim \Vert R_{u_h}  \Vert_{G^\ast } 
\end{align}
\end{lemma}

\begin{proof}
It follows from Lemma~\ref{lem:ulb} and the fact that
$$B(e , v) = B(u - u_h , v) = L(v) - B(u_h, v) = R_{u_h} (v)$$
for any $v\in V$.
\end{proof}

Let us introduce the parameters
\begin{align}
\tau_{K,0} :=\left( c_1 { \frac{k}{h^2} + c_3 s}\right)^{-1}, \quad \tau_{E,0} = c_4 \frac{1}{h} \tau_{K,0} \label{eq:alpha2}
\end{align}
which correspond to $\tau_K$ and $\tau_E$ for $a= 0$. 
In \cite{verfurth2005robust}, the parameter 
$$ \alpha = \min \left\lbrace h k^{-1/2} , s^{-1/2} \right\rbrace $$
is introduced (recall that we assume constant physical properties). In view of \eqref{eq:minab}, it follows that 
\begin{align}
\tau_{K,0}^{1/2} \sim \alpha   \label{eq:alfa-tau}
\end{align}

For any function $v\in V$, we have that
$$ \Vert \nabla (v - P_h v)\Vert_{\tilde{E}} \lesssim \varphi(h) \Vert \nabla v \Vert_{\tilde{E}} $$
with $\varphi \to 0$ as $h\to 0$, for any $E\in {\cal E}_h$. Thus, for $k> 0$ and a fixed $s\geq 0$, arbitrarily large, we may assume that $h$ is small enough, so that
\begin{align}
\varphi(h) \lesssim \frac{\frac{1}{h^2} k }{\frac{1}{h^2} k  + s}  = \frac{1}{h^2} k \tau_{K,0} \label{eq:hyp-lem-intest}
\end{align}

\begin{lemma}[Interpolation estimates]\label{lem:lem-intest}
Suppose that $h$ is small enough so that \eqref{eq:hyp-lem-intest} holds. Then:
\begin{align}
\Vert v - P_h v \Vert_K & \lesssim  \tau_{K,0}^{1/2} \Vert v \Vert_{G,\tilde{K}} \label{eq:ie1} \\
\Vert v - P_h v \Vert_E & \lesssim  \tau_{E,0}^{1/2} \Vert v \Vert_{G,\tilde{E}}  \label{eq:ie2} \\
\Vert P_h v \Vert_{G,K} & \lesssim \Vert v \Vert_{G,\tilde{K}}  \label{eq:ie3}
\end{align}
\end{lemma}
\begin{proof}
Bound \eqref{eq:ie1} follows from standard interpolation estimates, noting that $P_h$ enjoys the same approximation properties as the best interpolant.
This fact yields:
\begin{align}
\Vert v - P_h v \Vert_K & \lesssim \min \left\{ \frac{h}{k^{1/2}} k^{1/2}\Vert \nabla v \Vert_{\tilde{K}} , \frac{1}{s^{1/2}} s^{1/2} \Vert v \Vert_{\tilde{K}}\right\}\nonumber\\
&  \lesssim \min  \left\{  \frac{h}{k^{1/2}} ,  \frac{1}{s^{1/2}} \right\}  \Vert v \Vert_{G,\tilde{K}} \nonumber\\
&  \lesssim \tau_{K,0}^{1/2} \Vert v \Vert_{G,\tilde{K}} \el
\end{align}
Using \eqref{eq:tracein1}, for  \eqref{eq:ie2} we have that
\begin{align}
\Vert v - P_h v \Vert^2_E & \lesssim {h} \Vert \nabla (v - P_h v ) \Vert^2_{\tilde{E}} +  \frac{1}{h} \Vert v - P_h v  \Vert^2_{\tilde{E}} \label{eq:lem6-1}
\end{align}
For the second term, we have that
\begin{align}
\frac{1}{h} \Vert v - P_h v  \Vert^2_{\tilde{E}} 
\lesssim \frac{1}{h}  \min \left\{ \frac{h^2}{ k} , \frac{1}{s} \right\} \Vert v \Vert^2_{G,\tilde{E}}
 \lesssim \frac{1}{h}\tau_K \Vert v \Vert^2_{G,\tilde{E}}
 \lesssim \tau_E \Vert v \Vert^2_{G,\tilde{E}} \label{eq:lem6-2}
\end{align}
whereas for the first term in \eqref{eq:lem6-1} we have, assuming that \eqref{eq:hyp-lem-intest} holds:
\begin{align*}
{h} \Vert \nabla (v - P_h v ) \Vert^2_{\tilde{E}} \lesssim \frac{1}{h} k \tau_{K,0} \frac{1}{k}\Vert v \Vert^2_{G,\tilde{E}} =  \tau_{E,0} \Vert v \Vert^2_{G,\tilde{E}}
\end{align*}
This, together with \eqref{eq:lem6-2}, proves \eqref{eq:ie2}. Finally, bound \eqref{eq:ie3} is a consequence of the $H^1$-stability of $P_h$.
\end{proof}

\begin{lemma}[Introduction of bubble functions] \label{lem:bf}
There exists a bubble function $\psi_K$ with support in $K$ and a bubble function $\psi_E$ with support in $\tilde{E}$ such that, for all piecewise polynomial functions $v_h$, the following holds:
\begin{align}
\langle v_h , \psi_K v_h\rangle_K & \gtrsim \Vert v_h \Vert^2_K \label{eq:lem-bf-1} \\
\Vert \psi_K v_h \Vert_{G,K} & \lesssim \tau_{K,0}^{-1/2} \Vert v_h \Vert_K \label{eq:lem-bf-2} \\
\langle v_h , \psi_E v_h\rangle_E & \gtrsim \Vert v_h \Vert^2_E  \label{eq:lem-bf-3} \\
\Vert \psi_E v_h \Vert_{G,\tilde{E}} & \lesssim  \tau_{E,0}^{-1/2} \Vert v_h \Vert_{E}  \label{eq:lem-bf-4} 
\end{align}
\end{lemma}
\begin{proof}
In Lemma 3.6 of \cite{verfurth2005robust} it is shown, in essence, that functions $\psi_K v_h$ and $\psi_E v_h$ have the same boundedness properties as $v_h$ within each element and edge, respectively, but with local support. Note that they are polynomials (of degree higher than $v_h$), and, therefore, inverse estimates can be applied. Bounds \eqref{eq:lem-bf-1} and \eqref{eq:lem-bf-3} can be directly found in Lemma 3.6 of \cite{verfurth2005robust}. The proof of \eqref{eq:lem-bf-2} follows using the inverse estimate \eqref{eq:invest}:
\begin{align}
\Vert \psi_K v_h \Vert^2_{G,K} & \lesssim \left( \frac{k}{h^2} + s\right) \Vert \psi_K v_h \Vert^2_{K}
\lesssim \tau_{K,0}^{-1} \Vert v_h \Vert^2_K 
\end{align}
To prove  \eqref{eq:lem-bf-4}, we proceed similarly, using now \eqref{eq:maxE}:
\begin{align}
\Vert \psi_E v_h \Vert^2_{G,\tilde{E}} &\lesssim 
\tau_{K,0}^{-1} \Vert v_h \Vert^2_{\tilde{E}} \lesssim 
{h}\tau_{K,0}^{-1} \Vert v_h \Vert^2_{E} \lesssim 
\tau_{E,0}^{-1} \Vert v_h \Vert^2_{E} 
\end{align}
from where the Lemma follows.
\end{proof}

Let us define
\begin{align}
\eta^2_{K,0} & = \tau_{K,0} \Vert  P_h^\bot (R_K) \Vert^2_K + \sum_{E\subset \partial K} \tau_{E,0} \Vert R_E \Vert^2_E\nonumber\\
\eta_0 & = \left( \sum_K \eta^2_{K,0} \right)^{1/2}\el
\end{align}

\begin{thm}[A posteriori error estimates]\label{thm:ap-ver}
Under the assumptions of Lemma~\ref{lem:lem-intest}, for both the Galerkin and the OSGS methods, it holds that
\begin{align}
\eta_0 \lesssim \Vert e \Vert_S \lesssim  \eta_0 \label{eq:apost}
\end{align}
\end{thm}

\begin{proof}
{\bf Upper bound}:  We have that
\begin{align}
R_{u_h}(v) & = R_{u_h}(v - P_h v) +  R_{u_h}(P_h v) \el
\end{align}
Let us bound each term separately. Using Lemma~\ref{lem:lem-intest}:
\begin{align}
R_{u_h}(v - P_h v) & = \sum_K \langle P_h^\bot (R_K ), v - P_h v\rangle_K + \sum_E \langle R_E , v - P_h v\rangle_E\nonumber\\
& \lesssim \sum_K \Vert P_h^\bot ( R_K) \Vert_K \, \tau_{K,0}^{1/2} \, \Vert v\Vert_{G,\tilde{K}} +  \sum_E \Vert R_E \Vert_E  \, \tau_{E,0}^{1/2} \, \Vert v\Vert_{G,\tilde{E}}\nonumber\\ 
& \lesssim  \Vert v\Vert_{G} \left[   \sum_K  \tau_{K,0}\,\Vert P_h^\bot (R_K) \Vert^2_K + \sum_E \tau_{E,0} \, \Vert R_E \Vert_E^2 \right]^{1/2}\nonumber\\
& \lesssim  \Vert v\Vert_{G}  \eta_0\el
\end{align}
For the Galerkin method, $R_{u_h}(P_h v) = 0$. For the OSGS method, using the expressions of $\tau_K$ and $\tau_E$, we have that
\begin{align*}
R_{u_h}(P_h v) & = L(P_h v)  - B(u_h , P_h v ) \nonumber\\
& = \underbrace{L(P_h v) -  B_{\rm stab} (u_h , P_h v)}_{=0} + S(u_h, P_h v)\nonumber\\
& \lesssim  \sum_K \tau_K \Vert P_h^\bot (R_K) \Vert_K \, \left( \Vert k \Delta P_h v \Vert_K + \Vert a\cdot \nabla (P_h v)\Vert_K \right) \nonumber\\
&  \quad + \sum_E \tau_E k^2 \Vert \jump{\partial_n u_h} \Vert_E \, \Vert \jump{\partial_n P_h v} \Vert_E \nonumber\\
& \lesssim  \sum_K \tau_K \Vert P_h^\bot (R_K) \Vert_K  \left[ \frac{k^{1/2}}{h}k^{1/2} \Vert \nabla (P_h v)\Vert_K
 + \vert a \vert \min\left\lbrace \frac{1}{k^{1/2}} , \frac{1}{s^{1/2} h}\right\rbrace \Vert P_h v\Vert_{G,\tilde{K}}   \right]  \nonumber\\
&  \quad + \sum_E \tau_E \Vert R_E \Vert_E  \frac{1}{h^{1/2}} k \Vert \nabla  (P_h v)\Vert_{\tilde{E}} \nonumber\\
& \lesssim  \sum_K \left( \tau^{1/2}_K  +  \tau_K  \vert a \vert  \frac{1}{h}\tau_{K,0}^{1/2}\right) \Vert P_h^\bot (R_K) \Vert_K \, \Vert P_h v \Vert_{G,\tilde{K}} \nonumber\\
&  \quad + \sum_E \tau_E^{1/2} \Vert R_E \Vert_E  k^{1/2} \Vert \nabla  (P_h v)\Vert_{\tilde{E}} \nonumber\\
& \lesssim \Vert v\Vert_{G}  \eta_0\el
\end{align*}
Therefore, using Lemma~\ref{lem:error-residual}:
$$ \Vert e \Vert_S \lesssim \Vert R_{u_h}  \Vert_{G^\ast } \lesssim \eta_0$$

{\bf Lower bound}. Let us pick the particular function 
$$ w_h = \sum_K \tau_{K,0} \psi_K R_K + \gamma \sum_E \tau_{E,0} \psi_E R_E $$
with $\gamma$ to be determined. Since the support of $\psi_K R_K$ is $K$, and the support of $\psi_E R_E$ is $\tilde{E}$, with only two elements, the support of their products extends to a fixed number of elements, not depending on $h$. Therefore, using \eqref{eq:lem-bf-2} and \eqref{eq:lem-bf-4} of Lemma~\ref{lem:bf}:
\begin{align}
\Vert w_h \Vert^2_G & \lesssim  \sum_K  \tau_{K,0}^2 \Vert \psi_K R_K\Vert_{G,K}^2 
+  \gamma^2 \sum_E \tau_{E,0}^2 \Vert \psi_E R_E \Vert^2_{G,\tilde{E}} \nonumber\\
& \lesssim  \sum_K  \tau_{K,0} \Vert R_K\Vert_{K}^2 +  \gamma^2 \sum_E \tau_{E,0} \Vert R_E \Vert^2_{E}\nonumber\\
& \lesssim \eta_0^2 \el 
\end{align}
On the other hand, noting that $\psi_K R_K$ is zero on all edges, using \eqref{eq:lem-bf-1} and \eqref{eq:lem-bf-3} of Lemma~\ref{lem:bf}  and \eqref{eq:maxE} we have that
\begin{align}
R_{u_h} (w_h) & = \sum_K \langle R_K , w_h\rangle_K + \sum_E \langle R_E , w_h\rangle_E \nonumber\\
& =  \sum_K \left[  \tau_{K,0} \langle R_K , \psi_K R_K\rangle_K + \gamma \sum_{E\subset \partial K}  \tau_{E,0} \langle R_K ,  \psi_E R_E\rangle_K \right] 
+  \sum_E  \gamma \tau_{E,0} \langle R_E ,  \psi_E R_E\rangle_E \nonumber\\
& \gtrsim  \sum_K  \tau_{K,0} \Vert R_K\Vert_K^2 +  \gamma  \sum_E \tau_{E,0} \Vert R_E\Vert_E^2 
-  \gamma \sum_K  \sum_{E\subset \partial K} \tau_{E,0} \Vert R_K \Vert_K\, \Vert \psi_E R_E\Vert_K\nonumber\\
& \gtrsim   \sum_K  \tau_{K,0} \Vert R_K\Vert_K^2 +  \gamma  \sum_E \tau_{E,0} \Vert R_E\Vert_E^2 
-  \gamma \sum_K  \sum_{E\subset \partial K} \tau_{E,0} {h^{1/2}} \Vert R_K \Vert_K\, \Vert R_E\Vert_E \nonumber\\
& \gtrsim   \sum_K \tau_{K,0}\Vert R_K\Vert_K^2 +  \gamma  \sum_E \tau_{E,0} \Vert R_E\Vert_E^2 
-  \gamma \left[\sum_K \beta  \tau_{K,0} \Vert R_K\Vert_K^2 + \frac{1}{\beta}\sum_E \tau_{E,0}\Vert R_E\Vert^2_E\right]\el
\end{align}
where $\beta$ arises from Young's inequality. Taking $\beta$ sufficiently large and $\gamma$ sufficiently small, it follows that
\begin{align}
R_{u_h} (w_h) \gtrsim \eta_0^2 \el
\end{align}
From the two results obtained and Lemma 5, we have that
$$ \Vert e \Vert_S \gtrsim \Vert R_{u_h}  \Vert_{G^\ast } = \sup_{w\in V, w\not = 0} \frac{R_{u_h} (w)}{\Vert w \Vert_G}
 \geq \frac{R_{u_h} (w_h)}{\Vert w_h \Vert_G} \gtrsim \eta_0$$
which proves the theorem.
\end{proof}

\subsection{The proposed a posteriori error estimate}

Estimate $\Vert e \Vert_S\sim \eta_0$ means that
\begin{align}
& k\Vert \nabla e \Vert^2 + s \Vert e \Vert^2 + \sup_{w\in V, w\not = 0} \frac{(a\cdot\nabla e ,  w)^2}{ k\Vert \nabla w \Vert^2 + s \Vert w \Vert^2 }\nonumber\\
& \quad \sim \sum_K \frac{1}{\frac{1}{h^2}k +  s} \Vert P_h^\bot (-k \Delta u_h + a\cdot\nabla u_h + s u_h - f) \Vert_K^2 +
\sum_E  \frac{1}{\frac{1}{h}k  + {h}s} \Vert k\jump{\partial_n u_h}\Vert_E^2\label{eq:est-obtained}
\end{align}
For $s=0$, $k\to 0$, this estimates the $H^{-1}$-norm of the error of the convective term. This is a rather weak result, as we would wish to obtain an estimate in terms of the $L^2$-norm of the convective term, perhaps with mesh-dependent coefficients. This would be achieved if we could prove that
\begin{align}
& k\Vert \nabla e \Vert^2 + s \Vert e \Vert^2 +   \frac{1}{\frac{1}{h^2}k + \frac{1}{h}\vert a \vert +  s}\Vert a\cdot\nabla e \Vert^2 \nonumber\\
& \quad \sim \sum_K \frac{1}{\frac{1}{h^2}k + \frac{1}{h}\vert a \vert +  s} \Vert P_h^\bot (-k \Delta u_h + a\cdot\nabla u_h + s u_h - f) \Vert_K^2\nonumber\\
& \quad\quad  + \sum_E  \frac{1}{\frac{1}{h}k  + \vert a \vert + {h}s} \Vert k\jump{\partial_n u_h}\Vert_E^2\label{eq:est-wished}
\end{align}
For $s=0$, $k\to 0$, \eqref{eq:est-obtained} and \eqref{eq:est-wished} yield:
\begin{align}
\Vert a\cdot\nabla e \Vert_{H^{-1}} \sim {h} \Vert P_h^\bot  (a\cdot\nabla u_h  - f) \Vert \qquad \hbox{from \eqref{eq:est-obtained} }\nonumber\\
\Vert a\cdot\nabla e \Vert \sim \Vert P_h^\bot  (a\cdot\nabla u_h  - f) \Vert \qquad \hbox{from \eqref{eq:est-wished} }\el
\end{align}

In terms of the parameters $\tau_K$ and $\tau_E$, estimate \eqref{eq:est-wished} can be written as 
\begin{align*}
& k\Vert \nabla e \Vert^2 + s \Vert e \Vert^2 +   \tau_K \Vert a\cdot\nabla e \Vert^2 \\
& \quad \sim \sum_K \tau_K \Vert P_h^\bot (-k \Delta u_h + a\cdot\nabla u_h + s u_h - f) \Vert_K^2 +
\sum_E  \tau_E \Vert k\jump{\partial_n u_h}\Vert_E^2
\end{align*}
In view of the definition of the SGSs in Eqs.~\eqref{eq:SGSs1}--\eqref{eq:SGSs2} and the stabilized norm in Eq.~\eqref{eq:stab-norm}, the result we wished to prove is:
\begin{align}
\vertiii{e}^2 \sim \eta^2 := \sum_K \tau_K^{-1} \Vert u'_K\Vert_K^2 +  \sum_E \tau_E^{-1} \Vert u'_E\Vert_E^2 \label{eq:ap-prop}
\end{align}
That is to say, the {\em stabilized} norm of the error behaves as a {\em scaled} norm of the SGSs. With respect to the result proved in Theorem~\ref{thm:ap-ver}, the changes are that the norm $\Vert \cdot \Vert_S$ is replaced by $\vertiii{\cdot}$ and the parameters $\tau_{K,0}$ and $\tau_{E,0}$ by $\tau_K$ and $\tau_E$, respectively (and therefore $\eta_0$ by $\eta$).

{\em Our proposal is estimate \eqref{eq:ap-prop}.} We will see in the numerical examples that it behaves very well, but we cannot prove it analytically unless some continuity conditions on the convective term are assumed, which, in general, cannot be shown to hold. These conditions were implicitly assumed in \cite{codina2021posteriori} for the linearized Navier-Stokes equations, and explicitly stated in \cite{john-novo-2013}, where a result similar to \eqref{eq:ap-prop} was proposed for the SUPG method. 

\subsection{John-Novo's approach}\label{sec:john-novo}

In the following, we adapt the analysis presented in~\cite{john-novo-2013} for the SUPG method to our numerical formulation, that is to say, using the OSGS method and including the SGSs on the element boundaries. As we shall see, this analysis is not fully satisfactory for two main reasons: 
\begin{itemize}
\item It is based on an assumption that cannot be shown to hold in practice, as discussed later.
\item The upper and lower bounds are not only proportional to $\eta$ in \eqref{eq:ap-prop}, but also include terms that depend on the interpolation error of the exact solution $u$; these, however, can be expected to be small. 
\end{itemize}

\begin{hyp} \label{hyp-1}
For any $v\in V$ there holds:
\begin{align}
\tau^{-1}_K  \Vert P_h^\bot(v)\Vert^2_{{\cal T}_h}  + \tau^{-1}_E  \Vert P_h^\bot(v)\Vert^2_{{\cal E}_h} 
\lesssim  \vertiii{v}^2 \label{eq:hyp-cont}
\end{align}
\end{hyp}

Observe that \eqref{eq:ie1}-\eqref{eq:ie2} in Lemma~\ref{lem:lem-intest} imply
\begin{align}
\tau^{-1}_{K,0}  \Vert P_h^\bot(v)\Vert^2_{{\cal T}_h}  + \tau^{-1}_{E,0}  \Vert P_h^\bot(v)\Vert^2_{{\cal E}_h} 
\lesssim  \Vert {v}\Vert ^2_G\el
\end{align}
Thus, \eqref{eq:hyp-cont} can be understood as a generalization to the case we wish to analyze. However, we are not able to prove it, and this is why it is accepted as an assumption, as in \cite{john-novo-2013}. The validity of this assumption is discussed below.


\begin{lemma}[Quasi continuity of $B$ and $S$]\label{lem:for-thm40}
Suppose that Assumption~\ref{hyp-1} and the assumptions of Theorem~\ref{thm1} hold. Let $e=u-u_h$ be the error of the OSGS formulation. Then, for any $v\in V$ the following holds:
\begin{align}
& B(e, P_h^\bot(v))\ \lesssim \eta \vertiii{v}  \label{eq:semi-cont10}\\
& B(e, P_h^\bot(v))\ \lesssim \vertiii{e} \,   \vertiii{v}   \label{eq:semi-cont20}\\
& S(e,v) \lesssim \eta (\vertiii{v} + \zeta(i_v))\label{eq:semi-cont40} \\
& S(e,v) \lesssim (\vertiii{e} + \zeta(i_u)) (\vertiii{v} + \zeta(i_v))\label{eq:semi-cont30}
\end{align}
\end{lemma}

\begin{proof}
Using Lemma~\ref{lem3} and Assumption~\ref{hyp-1}  we have that
\begin{align*}
B(e, P_h^\bot(v)) & = L(P_h^\bot(v)) - B(u_h, P_h^\bot(v)) \nonumber\\
& = \sum_K \langle  P^\bot_h(R_K) , P^\bot_h(v)\rangle_K +  \sum_E \langle  R_E  , P^\bot_h(v)\rangle_E \nonumber\\
& \lesssim \sum_K \tau_K^{1/2} \Vert P^\bot_h(R_K)\Vert_K \tau_K^{-1/2}  \Vert P^\bot_h(v)\Vert_K 
+ \sum_E \tau_E^{1/2} \Vert R_E\Vert_E \tau_E^{-1/2}  \Vert P^\bot_h(v)\Vert_E  \nonumber\\
&  \lesssim \eta \vertiii{v} 
\end{align*}
which is \eqref{eq:semi-cont10}. Inequality \eqref{eq:semi-cont20} follows from the definition of $B$ and using Assumption~\ref{hyp-1} for the convective term:
$$ (a\cdot\nabla e , P_h^\bot(v)) \lesssim \tau_K^{1/2} \Vert a\cdot\nabla e \Vert \tau_K^{-1/2} \Vert P_h^\bot(v)\Vert \lesssim \vertiii{e}\,\vertiii{v}$$
To prove \eqref{eq:semi-cont30}, let us start noting that, for any function $v\in V$:
\begin{align*}
\tau_K^{1/2} k \Vert \Delta v\Vert_K & \leq \tau_K^{1/2} \left( k \vert i_v\vert_{H^2(K)} + k \Vert \Delta P_h v\Vert_K\right) \nonumber\\
& \leq \tau_K^{1/2}  k \vert i_v\vert_{H^2(K)} + \tau_K^{1/2}  k C_{\rm inv} \frac{1}{h} \Vert \nabla P_h v\Vert_K \nonumber\\
& \lesssim \tau_K^{1/2}  k \vert i_v\vert_{H^2(K)} + k^{1/2}\Vert \nabla P_h v\Vert_K \el
\end{align*}
Hence, since $i_e = e - P_h(e) = u - P_h u = i_u$:
\begin{align*}
S(e,v) & \lesssim  \sum_K (  \tau^{1/2}_Kk \vert i_v \vert_{H^2(K)} 
+ k^{1/2}\Vert \nabla v\Vert_K 
+ \tau^{1/2}_K \Vert  a\cdot\nabla v   \Vert_K 
+ s^{1/2} \Vert  v \Vert_K ) \nonumber\\
& \quad \times (  \tau^{1/2}_K k \vert i_u \vert_{H^2(K)} +  k^{1/2}\Vert \nabla e \Vert_K +  \tau^{1/2}_K  \Vert  a\cdot\nabla e   \Vert_K + s^{1/2} \Vert  e \Vert_K ) \nonumber\\
& \quad + \sum_E \tau_E \Vert k \jump{\partial_n i_v }\Vert_E \, \Vert \jump{\partial_n i_u }\Vert_E \nonumber\\
&  \lesssim (\vertiii{e} + \zeta(i_u)) ( \vertiii{v} + \zeta(i_v))
\end{align*}
which is \eqref{eq:semi-cont30}. Finally, \eqref{eq:semi-cont40} can be proved proceeding as before and using the fact that $S(u,v) = 0$ (recall that we consider $f\in V_h$):
\begin{align*}
S(e,v)   & = - S(u_h  , v) \nonumber\\
&  =  \sum_K \tau_K \langle  k \Delta v + a\cdot\nabla v - s v ,  P_h^\bot ( R_K)\rangle_K 
 + \sum_E \tau_E  \left\langle  k\jump{\partial_n v} , R_E\right\rangle_E \nonumber\\
& \lesssim \eta (\vertiii{v} + \zeta(i_v))
\end{align*} 
This completes the proof of the Lemma.
\end{proof}

\begin{thm}[A posteriori error estimate in the stabilized norm]\label{eq:ee-ub}
Under the assumptions of Lemma~\ref{lem:for-thm40}, there holds:
\begin{align}
\eta - \zeta (i_u) \lesssim \vertiii{e}  \lesssim \eta + \zeta  (i_u) 
\end{align}
\end{thm}

\begin{proof} {\bf Upper bound}: From Theorem~\ref{thm1} we have that
\begin{alignat}{3}
\vertiii{e}^2 - \zeta^2 (i_u)&  \lesssim B_{\rm stab}(e, \Lambda(e)) \quad && \quad  \hbox{}\nonumber\\
& = B_{\rm stab}(e, e) + {B_{\rm stab}(e, \tilde{e}_h)}\quad &&\quad  \hbox{from the definition of  $\Lambda(e)$} \nonumber\\
& = B_{\rm stab}(e, P^\bot_h(e)) \quad && \quad \hbox{from consistency}\nonumber\\
& = B(e, P^\bot_h(e)) + S(e,P^\bot_h(e))&&\quad  \hbox{from the definition of $B_{\rm stab}$}  \label{eq:thm-4-1}
\end{alignat}
Applying \eqref{eq:semi-cont10} and  \eqref{eq:semi-cont40} for $v=P^\bot_h(e) $ and noting that $P^\bot_h(e) = P^\bot_h(u) = i_u$ we obtain:
\begin{align*}
\vertiii{e}^2 - \zeta^2(i_u)  
& \lesssim  \eta ( \vertiii{P^\bot_h (e)} + \zeta(i_u))\nonumber\\
& \lesssim  \alpha  \eta^2 + \frac{1}{\alpha} \zeta^2 (i_u) +  \frac{1}{\alpha} \vertiii{e}^2 
\end{align*}
The upper bound follows by taking $\alpha$ large enough. 

{\bf Lower bound}: We proceed similarly to the proof of Theorem~\ref{thm:ap-ver}. The function $w_h$ we pick is the same, just replacing $\tau_{K,0}$ by $\tau_{K}$ and 
$\tau_{E,0}$ by $\tau_{E}$, i.e.,
$$ w_h = \sum_K \tau_{K} \psi_K R_K + \gamma \sum_E \tau_E \psi_E R_E $$
with $\gamma$ to be determined. As in Theorem~\ref{thm:ap-ver}, we have that
\begin{align}
& \vertiii{w_h}^2  \lesssim  \sum_K  \tau_{K}^2 \vertiii{\psi_K R_K}_{K}^2 
+  \gamma^2 \sum_E \tau_E^2 \vertiii{\psi_E R_E}^2_{\tilde{E}} \nonumber\\
& \quad \lesssim  \sum_K  \tau_{K}^2 \left( k \frac{1}{h^2} + \tau_K \frac{\vert a\vert^2}{h^2} +  s\right) \Vert \psi_K R_K\Vert_{K}^2  
+  \gamma^2 \sum_E  \tau_{E}^2 \left( k \frac{1}{h^2} + \tau_K \frac{\vert a\vert^2}{h^2} +  s\right) \Vert \psi_E R_E\Vert_{\tilde{E}}^2 \nonumber\\
& \quad \lesssim  \sum_K  \tau_{K} \Vert {R_K}\Vert_{K}^2 +  \gamma^2 \sum_E \tau_E^2 \tau_K^{-1}{h} \Vert R_E \Vert^2_{E}\nonumber\\
& \quad \lesssim \eta^2 \el 
\end{align}
To arrive at this result, we need to use inverse estimates and the fact that $\tau_K\sim \tau_E {h}$, as well as \eqref{eq:maxE}. Note that $\psi_K R_K$ and $\psi_E R_E$ are polynomials.

On the other hand, noting that $\psi_K R_K$ is zero on all edges, using Lemma~\ref{lem3} we have that
\begin{align}
& \vert B(e,w_h)\vert  = \vert B(u,w_h) - B(u_h,w_h)\vert = \vert L(w_h) - B(u_h,w_h)\vert = \vert R_{u_h} (w_h)\vert \nonumber\\
& \quad = \left\vert \sum_K \langle R_K , w_h\rangle_K + \sum_E \langle R_E , w_h\rangle_E\right\vert  \nonumber\\
& \quad =  \left\vert \sum_K \left[  \tau_{K} \langle R_K , \psi_K R_K\rangle_K + \gamma \sum_{E\subset \partial K}  \tau_E \langle R_K ,  \psi_E R_E\rangle_K \right] 
+  \sum_E  \gamma \tau_E \langle R_E ,  \psi_E R_E\rangle_E \right\vert \nonumber\\
&\quad  \gtrsim  \sum_K  \tau_{K} \Vert R_K\Vert_K^2 +  \gamma  \sum_E \tau_E \Vert R_E\Vert_E^2 
-  \gamma \sum_K  \sum_{E\subset \partial K} \tau_E \Vert R_K \Vert_K \Vert \psi_E R_E\Vert_K\nonumber\\
& \quad  \gtrsim   \sum_K  \tau_{K} \Vert R_K\Vert_K^2 +  \gamma  \sum_E \tau_E \Vert R_E\Vert_E^2 
-  \gamma \sum_K  \sum_{E\subset \partial K} \tau_E {h^{1/2}} \Vert R_K \Vert_K \Vert R_E\Vert_E \nonumber\\
&\quad  \gtrsim   \sum_K \tau_{K}\Vert R_K\Vert_K^2 +  \gamma  \sum_E \tau_E \Vert R_E\Vert_E^2 
-  \gamma \left[\sum_K \beta  \tau_{K} \Vert R_K\Vert_K^2 + \frac{1}{\beta}\sum_E \tau_E\Vert R_E\Vert^2_E\right]\nonumber\\
&\quad \gtrsim \eta^2 
\end{align}
where $\beta$ sufficiently large and $\gamma$ sufficiently small have been taken in the last step. Once again, we have used \eqref{eq:maxE} and that $\tau_K\sim \tau_E {h}$. We can now apply Lemma~\ref{lem:for-thm40} for $v=w_h$, noting that we can take $\zeta(i_{w_h}) = 0$, because, despite not belonging to the FE space, $w_h$ is a polynomial, and we can bound its Laplacian using an inverse estimate. Therefore, from the previous inequality and Lemma~\ref{lem:for-thm40} we have that
\begin{align*}
\eta^2 & \lesssim  \vert  B(e, w_h) \vert  \nonumber\\
& \lesssim \vert B_{\rm stab}(e,w_h) \vert + \vert S(e, w_h)\vert \nonumber\\
& \lesssim \vert B_{\rm stab}(e,P^\bot_h w_h) \vert + \vert S(e, w_h)\vert \nonumber\\
& \lesssim \vert B(e,P^\bot_h w_h) \vert + \vert S(e, P^\bot_h w_h)\vert + \vert S(e, w_h)\vert  \nonumber\\
& \lesssim (\vertiii{e} + \zeta(i_u))  \vertiii{w_h} \nonumber\\
& \lesssim (\vertiii{e} + \zeta(i_u))  \eta\el
\end{align*}
from which the lower bound follows.
\end{proof}

\begin{rem} The lower bound can in fact be localized element-wise, as it is done in \cite{john-novo-2013}.
\end{rem}

\subsection*{On the significance of Assumption~\ref{hyp-1}}

The natural way to bound the convective term is
\begin{align}
(a\cdot\nabla u , P_h^\bot(v) ) & \leq \tau_K^{1/2} \Vert a\cdot\nabla u \Vert  \tau_K^{-1/2}  \Vert P_h^\bot(v)\Vert \nonumber\\
& \leq \vertiii{u}  \tau_K^{-1/2}  \Vert P_h^\bot(v)\Vert \el
\end{align}
We wish to understand the implications of assumption \eqref{eq:hyp-cont}, considering only the first term involving norms over elements. We have that
\begin{align}
\tau_K^{-1}  \Vert P_h^\bot(v)\Vert^2 & \lesssim \left( \frac{k}{h^2} + \frac{\vert a \vert}{h} + s\right)   \Vert P_h^\bot(v)\Vert^2  \nonumber\\
& \lesssim \max\left\{\frac{k}{h^2} , \frac{\vert a \vert}{h} , s \right\} \Vert P_h^\bot(v)\Vert^2  \el
\end{align}
Let us discuss the different scenarios we can encounter:
\begin{itemize}
\item If  $\max\left\{\frac{k}{h^2} , \frac{\vert a \vert}{h} , s \right\} = \frac{k}{h^2}$ (diffusion-dominated case):
$
\tau^{-1}_K  \Vert P_h^\bot(v)\Vert^2  \lesssim  \frac{k}{h^2} h^2 \Vert \nabla v \Vert^2 \lesssim \vertiii{v}^2
$. This is always true.
\item If  $\max\left\{\frac{k}{h^2} , \frac{\vert a \vert}{h} , s \right\} = s$ (reaction-dominated case):
$
\tau^{-1}_K  \Vert P_h^\bot(v)\Vert^2  \lesssim  s \Vert  v \Vert^2 \lesssim \vertiii{v}^2
$. This is always true.
\item If  $\max\left\{\frac{k}{h^2} , \frac{\vert a \vert}{h} , s \right\} =  \frac{\vert a \vert}{h}$ (convection-dominated case). Suppose that $d=2$, for simplicity, and let $a_\bot$ be a vector orthogonal to $a$ with the same modulus. We have that
\begin{align}
\tau^{-1}_K  \Vert P_h^\bot(v)\Vert^2 & \lesssim \frac{\vert a \vert}{h}  \frac{h^2}{\vert a \vert^2} \left( \Vert a\cdot \nabla v\Vert^2 +  \Vert a_\bot \cdot \nabla v\Vert^2 \right)\nonumber\\
 & \lesssim \frac{1}{\frac{k}{h^2} + \frac{\vert a \vert}{h} + s} \left( \Vert a\cdot \nabla v\Vert^2 +  \Vert a_\bot \cdot \nabla v\Vert^2 \right)\nonumber\\
 & \lesssim \tau_K  \left( \Vert a\cdot \nabla v\Vert^2 +  \Vert a_\bot \cdot \nabla v\Vert^2 \right)\el
\end{align}
\end{itemize}
Therefore, the only conflictive case is the convection-dominated one, and the assumption we need to use holds if
\begin{align}
\Vert a_\bot \cdot \nabla v\Vert^2 \lesssim  \Vert a\cdot \nabla v\Vert^2\el
\end{align}
or, alternatively,
\begin{align}
\Vert  \nabla v\Vert^2 \lesssim \frac{1}{\vert a \vert^2} \Vert a\cdot \nabla v\Vert^2\el
\end{align}
when convection is dominant. It means that all derivatives can be bounded in terms of the streamline derivative, a reasonable situation in convection-dominated flows but possibly unprovable. 

\section{Numerical results}

In this section, we provide some numerical examples to evaluate the performance of the OSGS method and examine the behavior of the APEE. In the first example, we propose a manufactured solution to analyze the performance of the APEE in a convection-dominated case on a simple square-shaped domain. In the second problem, we assess the diffusion-dominated case analogously using the same manufactured solution. For the third test case, we choose a different manufactured solution, involving large gradients. Finally, we analyze the behavior of the APEE in an L-shaped benchmark problem. For this example, we do not have an exact solution; instead, we considered a solution of a very fine mesh as the reference solution for computing the error. The numerical results are compared with the ASGS method in all the examples. Recall that both the stabilized formulation and the APEE include the contribution of the SGSs in the element interiors and on the edges. 

We assess the performance of the APEE in terms of the effectivity index, which is introduced to evaluate the quality and accuracy of the APEE. The effectivity index is denoted as $\mathcal{I}_{\mathrm{eff}}$, and it is defined as the ratio of the estimated and exact error:
\begin{equation*}
    \mathcal{I}_{\mathrm{eff}} = \frac{\text{Estimated error}}{\text{Exact error}} = \frac{\eta}{\vertiii{e}}
\end{equation*}
where $\eta=\left( \sum_K{\eta_K^2}\right)^{1/2}$ is the global error estimator. Ideally, the effectivity index should be close to unity, meaning that the APEE accurately captures the error. To evaluate the method's error convergence, we compute the error $e$ in the $L^2$ norm ${\left\|e\right\|}$ and in the stabilized norm $\vertiii{e}$ for different mesh sizes, normalized by the respective norm of the reference solution
 \begin{equation*}
     {\left\|e\right\|}= \frac{\left\|u-u_h\right\|}{\left\|u\right\|}, \quad 
     \vertiii{e}^2 =\frac{ {k {\left\|\nabla (u-u_h) \right\|}^2+\ s{\left\| u- u_h \right\|}^2+\tau_{K} {\left\|a \cdot \nabla (u-u_h)\right\|}^2}}{{k {\left\|\nabla u\right\|}^2+\ s{\left\| u \right\|}^2+\tau_{K} {\left\|a \cdot \nabla u \right\|}^2}}
 \end{equation*}
where $u$ and $u_h$ represent the exact (or sufficiently converged) and FE solution, respectively. 

\subsection{Convection-dominated problem} \label{ex1}

The first numerical example is intended to assess the performance of the APEE for a convection-dominated problem. Here, the contribution of the SGSs on the element edges in both the APEE and in the stabilized formulation is negligible because the diffusion term is very small. Hence, it is an ideal setting to observe the performance of the SGSs in the element interiors as an error estimator. The exact solution is represented by a two-dimensional polynomial similar to the one presented in~\cite{john2000numerical}. 
The computational domain has been chosen as the unit square, i.e., $\Omega = (0,1)\times (0,1)$. We have taken $k =10^{-5}$, $a=[0.4,0.7]$, $s=10^{-5}$, and the forcing term $f$ on the right-hand side of \eqref{modprob} has been chosen such that
\begin{equation*}
    u=100(1-x)^2 x^2 y (1-2y) (1-y)
\end{equation*}
is the exact solution (also known as the manufactured solution), where $x$ and $y$ are Cartesian coordinates. Note that this solution vanishes on $\partial\Omega$. Observe also that $f \not \in V_h$, as we have assumed to simplify the analysis. In this case, $f$ has to be included in the element-wise residuals $R_K$.

We have discretized $\Omega$ with a uniform bilinear quadrilateral mesh.
The rate of convergence of the error in the $L^2$ norm has been found to be $h^{p+1} = h^2$ ($p=1$), which is the optimal rate of convergence for linear elements. The error convergence rate for the stabilized norm that has been obtained is $h^{p+1/2} = h^{3/2}$, which confirms that the problem is convection-dominated. 
It is remarkable that the convergence of the VMS-based global error estimator $\eta$ that we computed is very similar to that of the error in the stabilized norm, showing an excellent agreement between the estimated and true error in this norm. The effectivity index converges to $1$, indicating that the APEE has precisely recovered the error. 
Fig.~\ref{ex1errcon} shows the error convergence rate and the effectivity indices for different mesh sizes for both OSGS and ASGS methods. It can be observed that the effectivity index remains constant when varying the mesh size. It also shows the convergence of the APEE and the stabilized error norm at the same rate.

In Fig.~\ref{ex1etaplot}, the VMS-based local error estimator $\eta_{K}$ is compared with the stabilized error norm $\vertiii{e}_K$ for each element $K$ on an exemplary mesh comprising $20 \times 20$ bilinear elements. While the values do not match exactly, they show closely related results both for the OSGS and ASGS methods. 
 
\begin{figure}[h]
    \centering
    \includegraphics[width=0.41\textwidth]{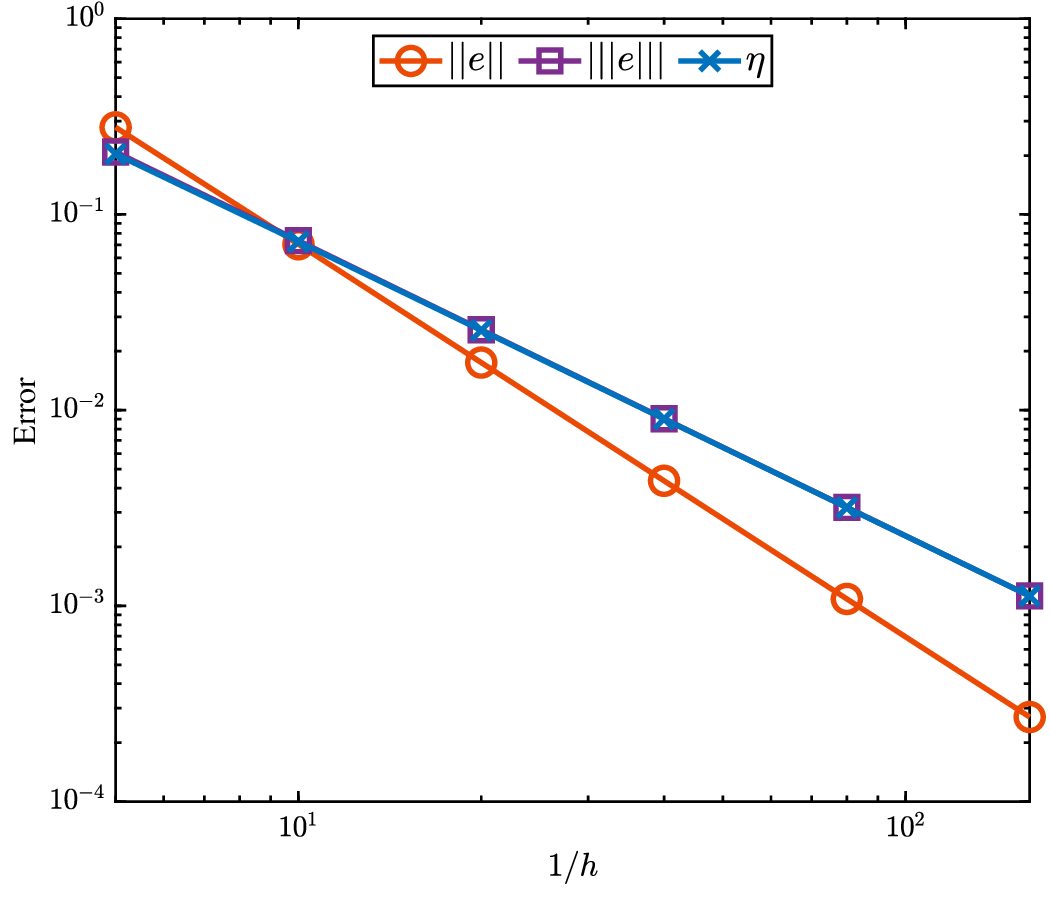}
    \includegraphics[width=0.41\textwidth]{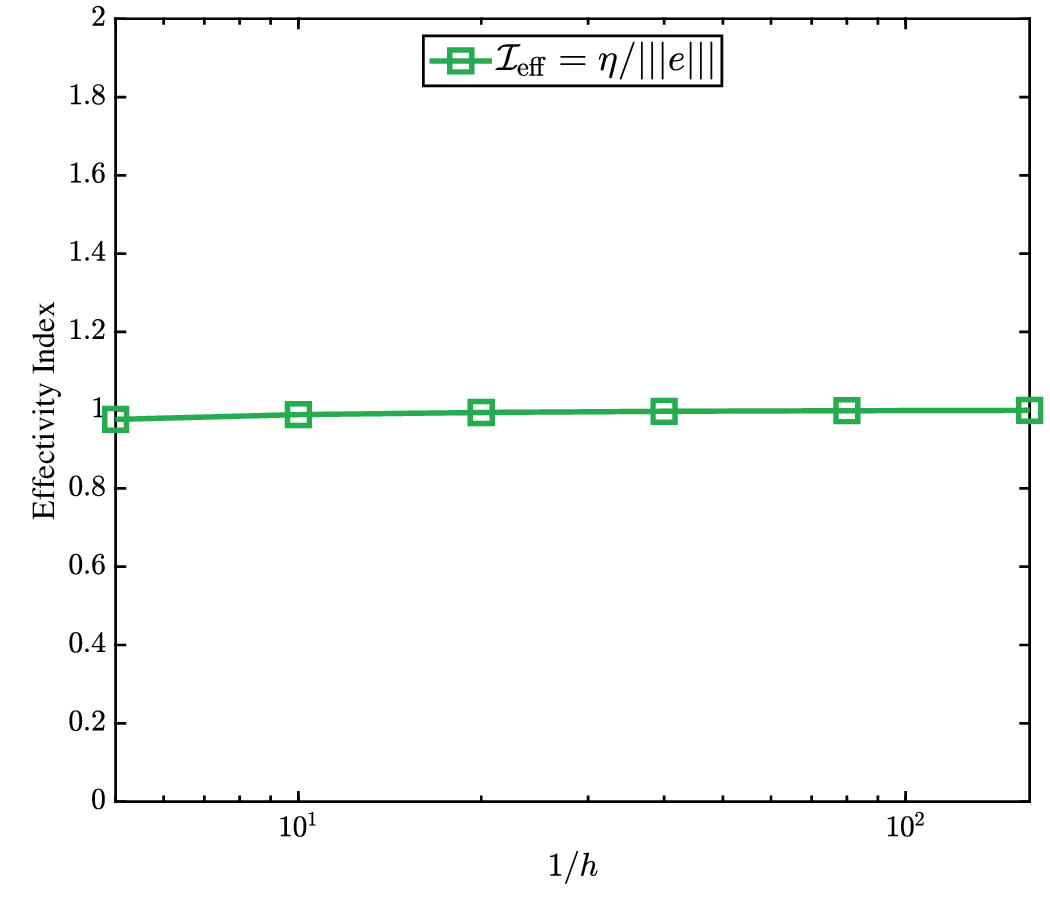}
    \includegraphics[width=0.41\textwidth]{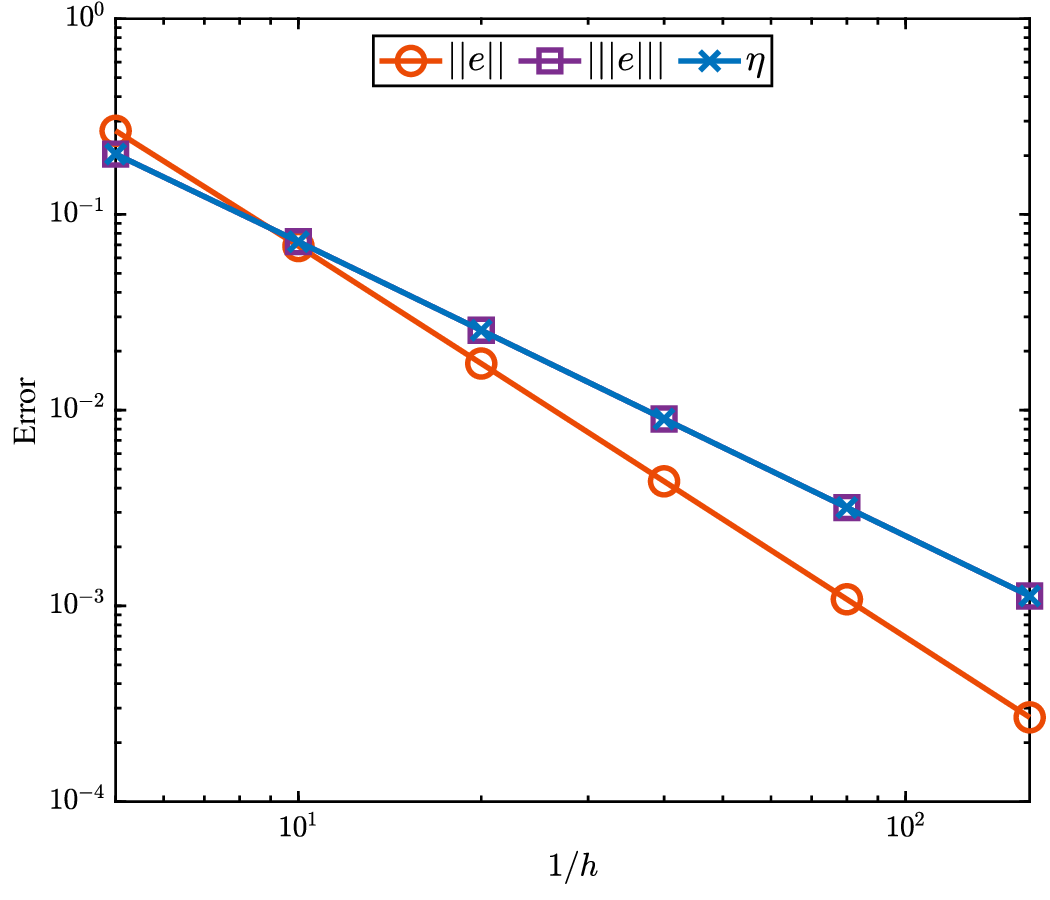}
    \includegraphics[width=0.41\textwidth]{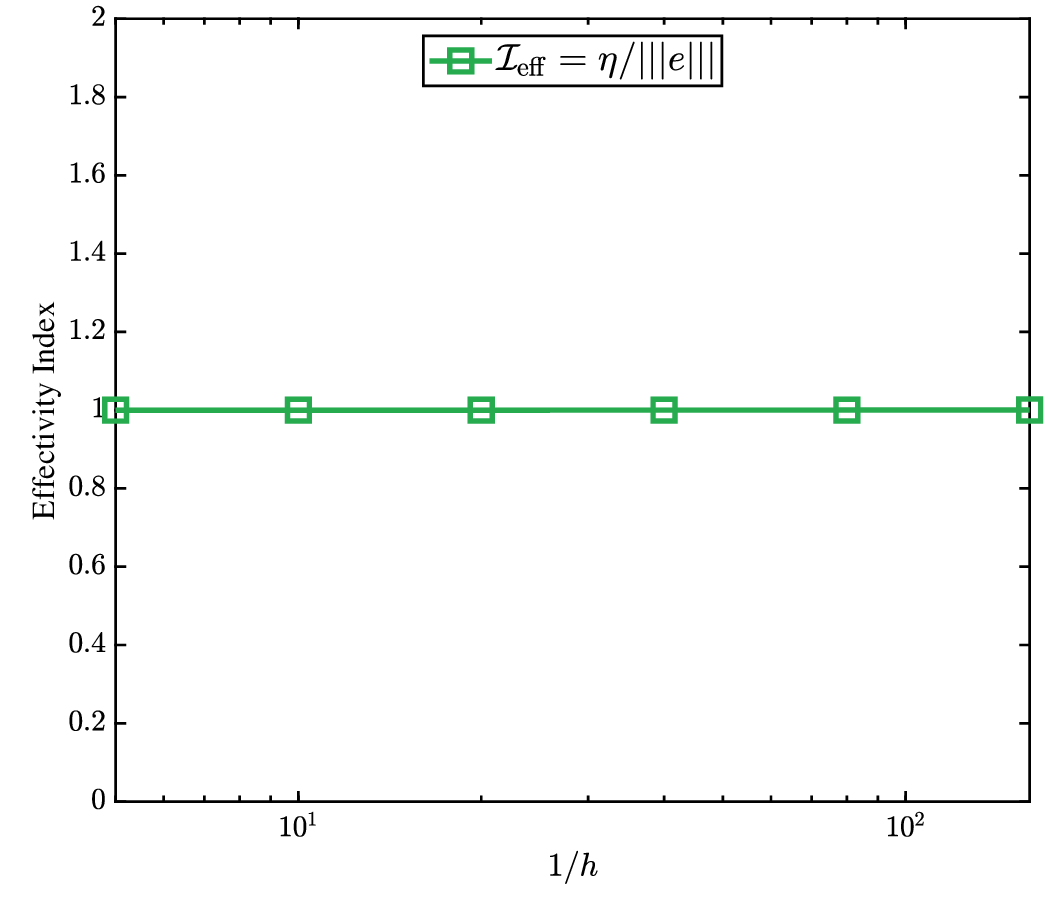}
    \caption{Results of example~\ref{ex1}. The figures on the left represent the error convergence in the $L^2$ norm, the stabilized norm and the APEE, the ones in the right represent the global effectivity index. Top: OSGS method; bottom: ASGS method.}
    \label{ex1errcon}
\end{figure}

\begin{figure}[h]
    \centering
    \includegraphics[width=0.41 \textwidth]{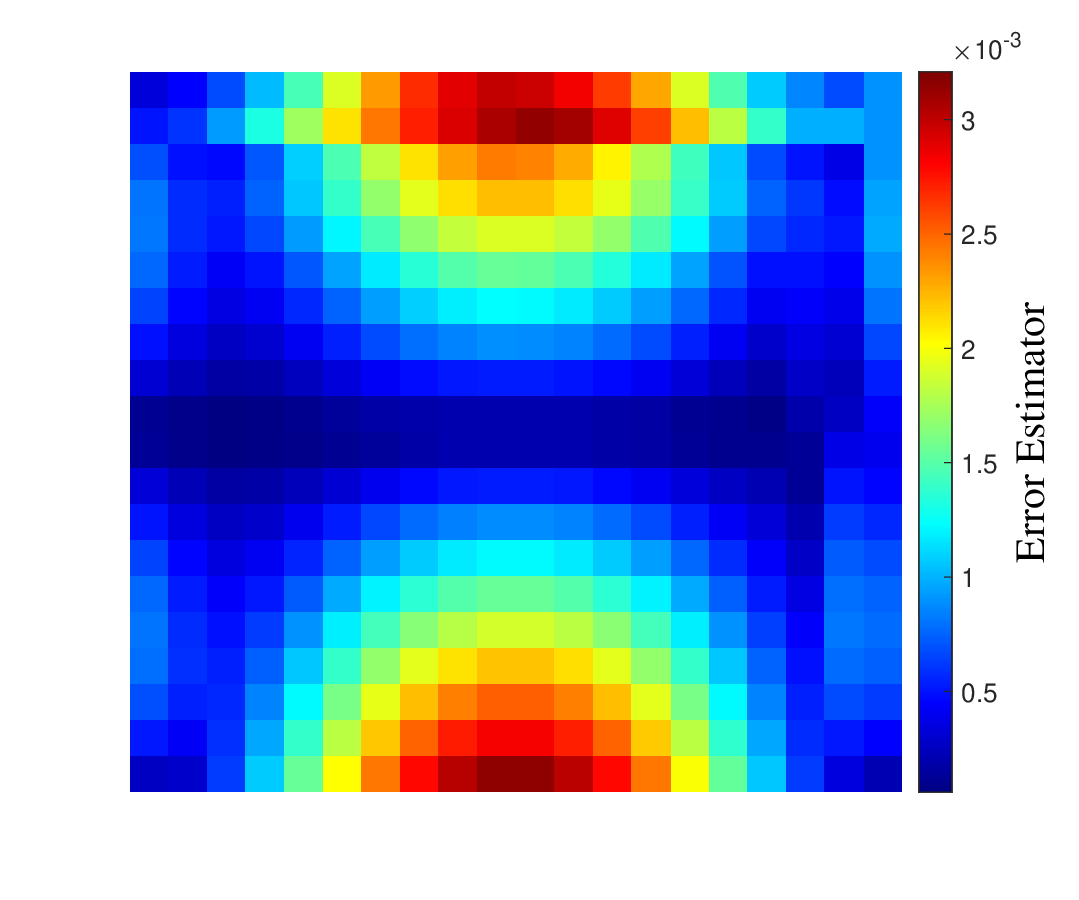}
    \includegraphics[width=0.41 \textwidth]{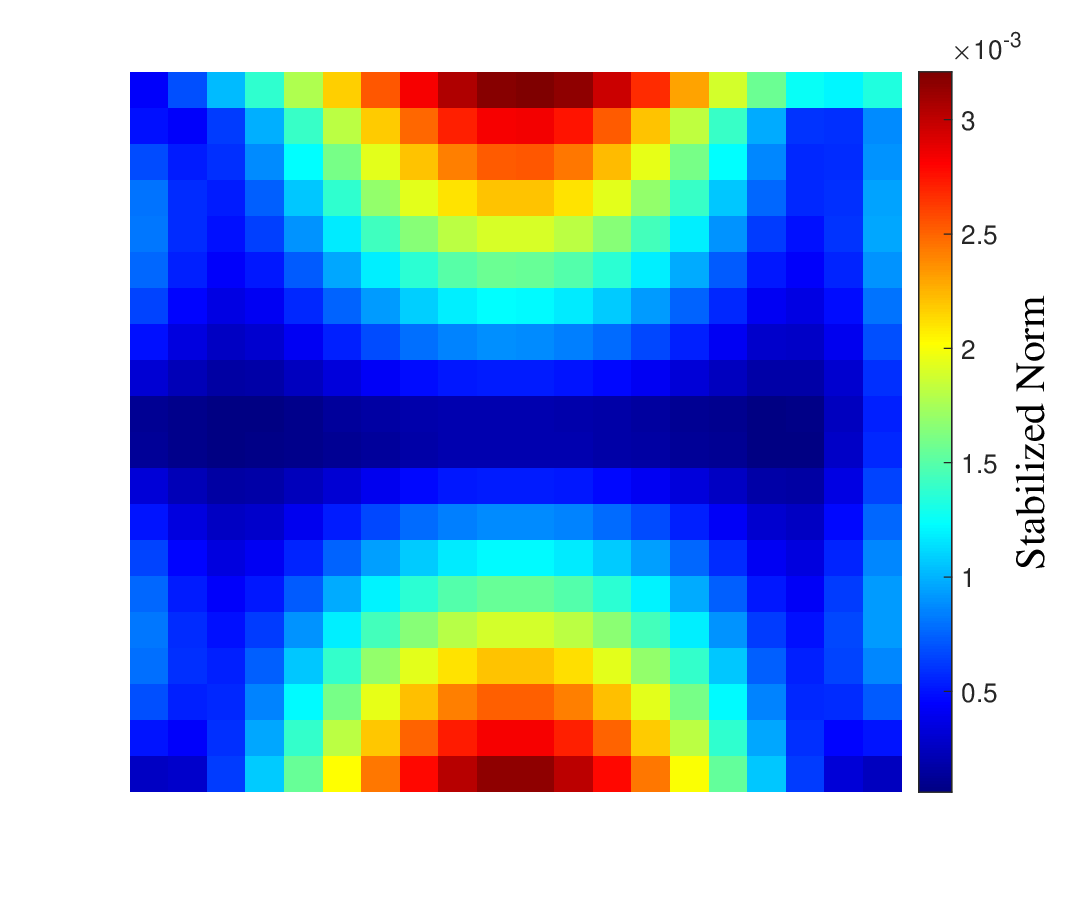}
    \includegraphics[width=0.41 \textwidth]{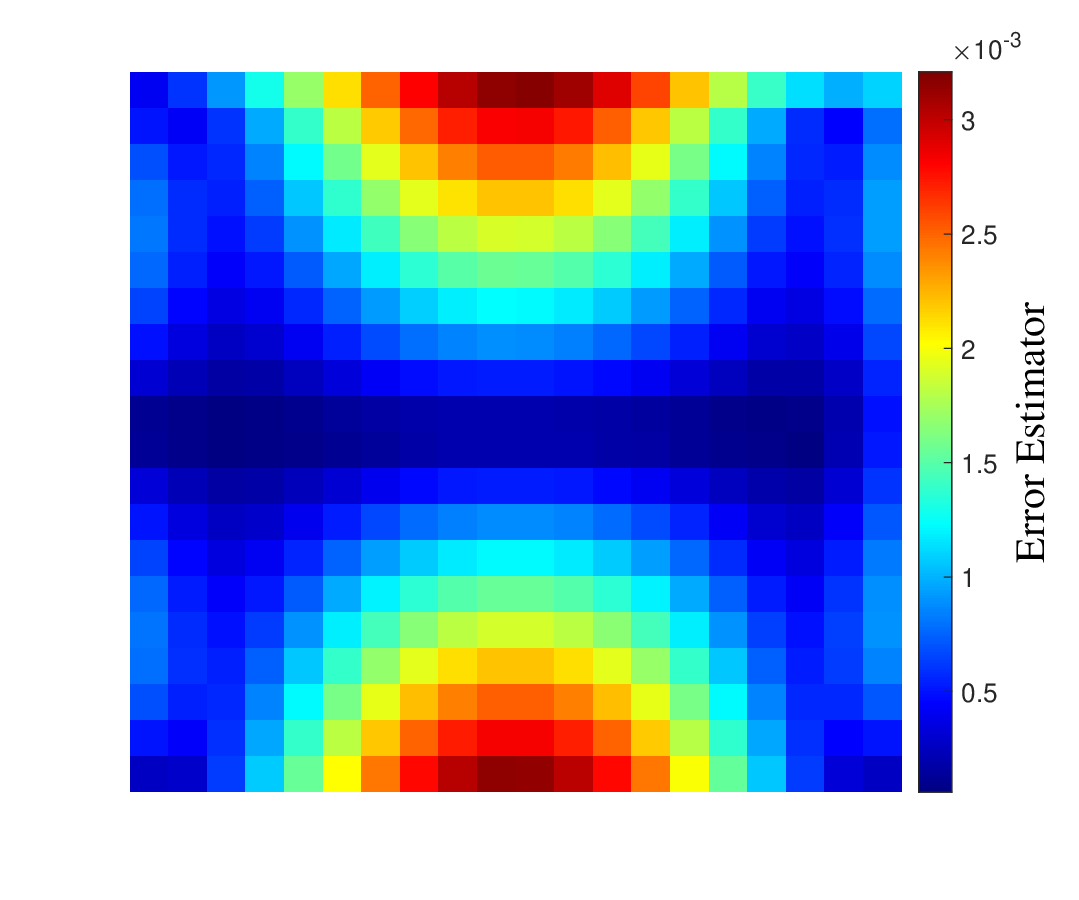}
    \includegraphics[width=0.41 \textwidth]{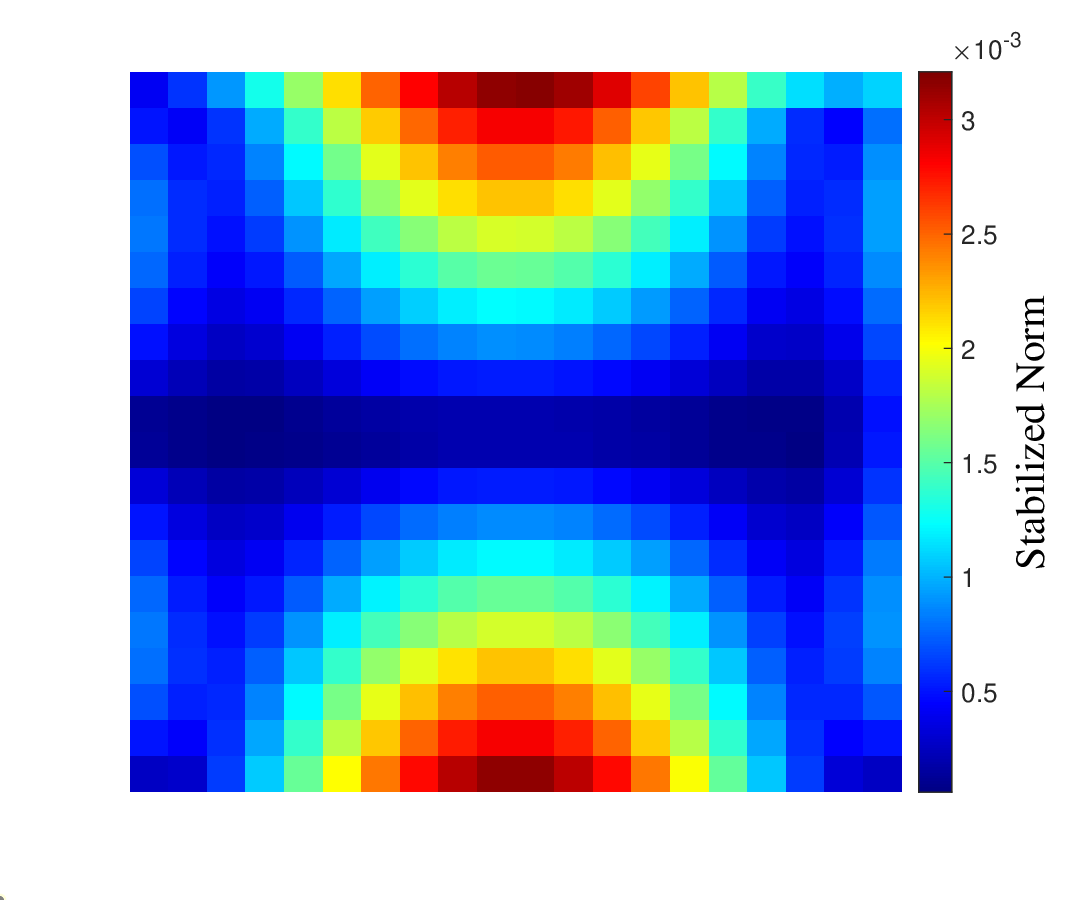}
    \caption{Results of example~\ref{ex1}. The figures on the left present the SGSs error estimator and on the right the stabilized norm error contribution. Top: OSGS method; bottom: ASGS method.}
    \label{ex1etaplot}
\end{figure}

\subsection{Diffusion-dominated problem}\label{example-2}

In this numerical example, we have tested the diffusion-dominated case. Here, the contribution of the SGSs on the inter-element boundaries is crucial. We chose $k =1$, $a=10^{-5}\,[0.4,0.7]$, $s=10^{-5}$, and the forcing term $f$ is determined by using the same manufactured solution as in Example~\ref{ex1}. The computational domain $\Omega$ is also the same. 
From the difference between the two finest meshes in Fig.~\ref{ex2errcon}, we obtain the optimal error convergence rates of $2.0$ in the $L^2$ norm, $1.0$ in the stabilized norm, and $1.0$ for the global error estimator. Note that, for a diffusion-dominated problem, the theoretical convergence rate of $h^p = h$ ($p=1$) is expected. The global error estimator converges at the same convergence rate. 
The effectivity index converges to $1.1$ and $2.3$ for the OSGS and ASGS methods, respectively. It is interesting to see that the effectivity index of the OSGS method is much closer to one, indicating a more accurate error estimation. 

Again, we analyze the behavior of the local error estimator and the stabilized error norm in each element. A comparison of the local error estimator and the stabilized error contributions of the OSGS and ASGS methods is presented in Fig.~\ref{ex2etaplot}. It is observed that the OSGS-based error estimator yields a better match with the stabilized error norm than the ASGS method. For the results in this figure we have again utilized a uniform mesh of $20 \times 20$ quadrilateral elements. It is remarkable that in this case the main difference between OSGS and ASGS is the use of $P^\bot_h(f)$ or just $f$ in the residual. This shows the convenience of considering the SGSs in the element interiors as orthogonal to the FE space.

\begin{figure}[h]
    \centering
    \includegraphics[width=0.41 \textwidth]{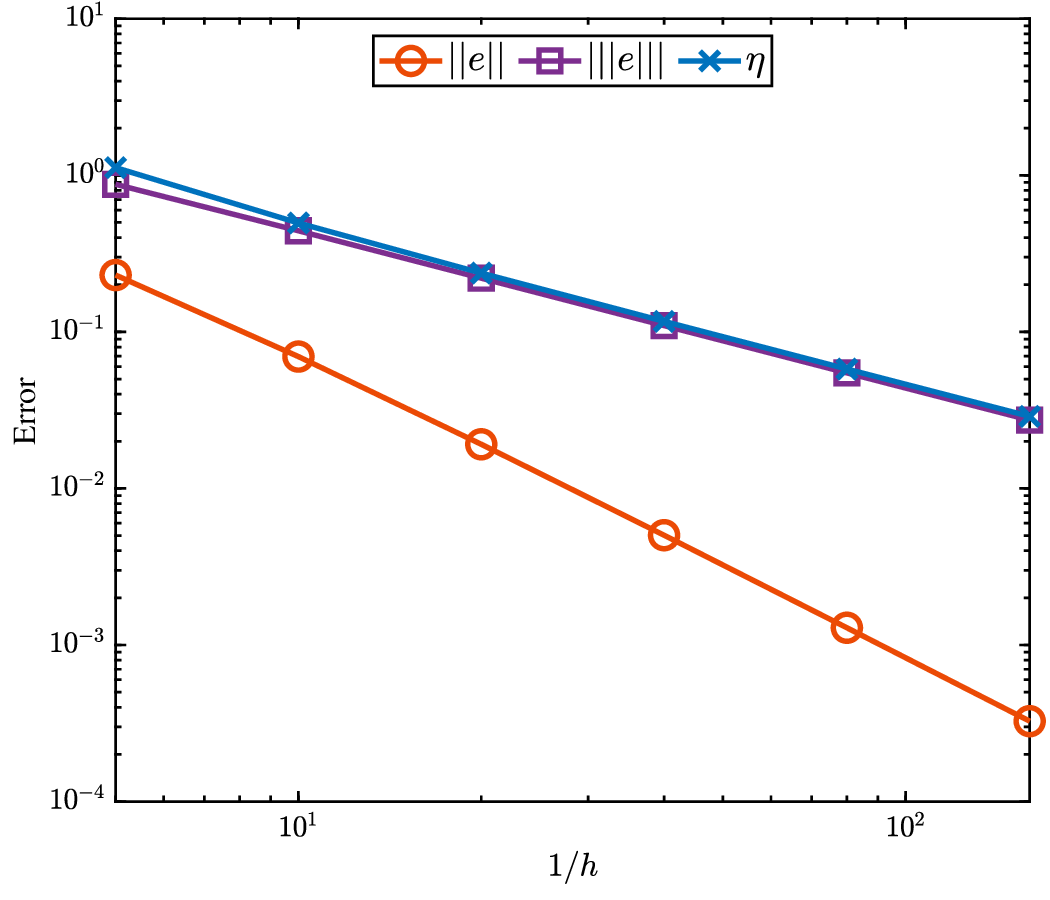}
    \includegraphics[width=0.41 \textwidth]{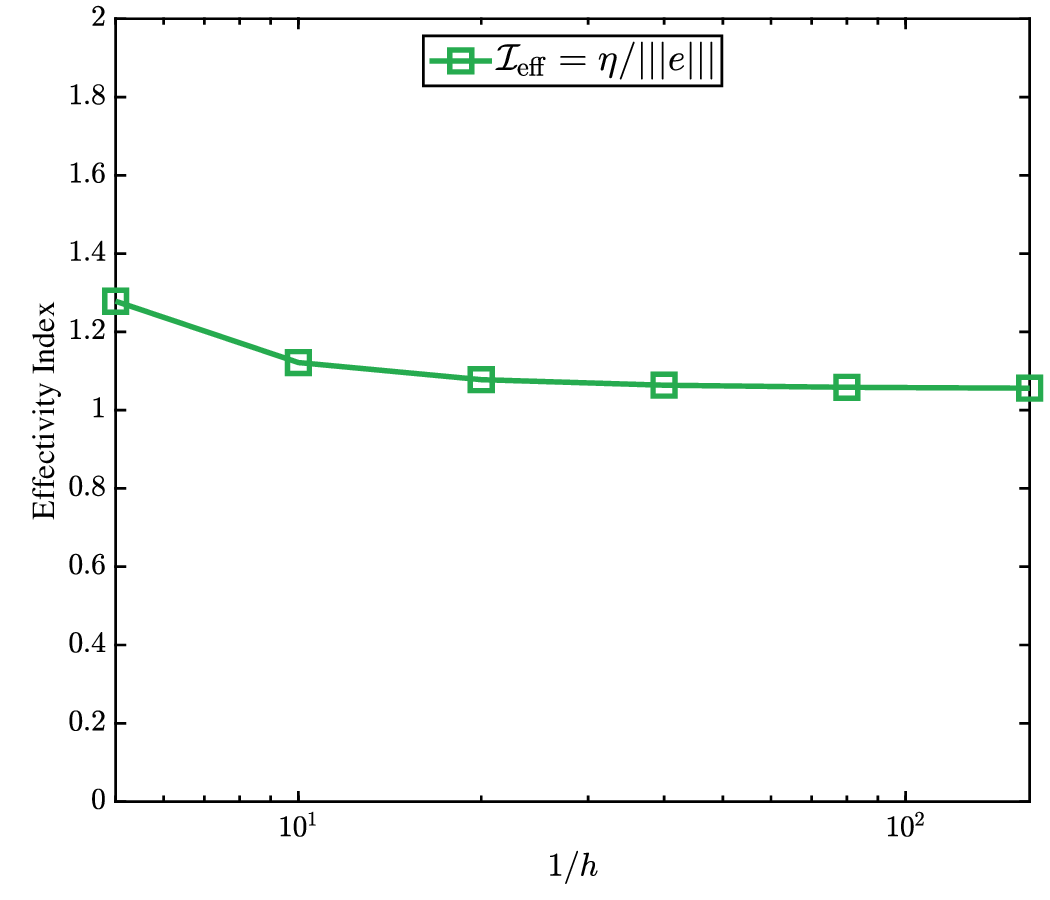}
    \includegraphics[width=0.41 \textwidth]{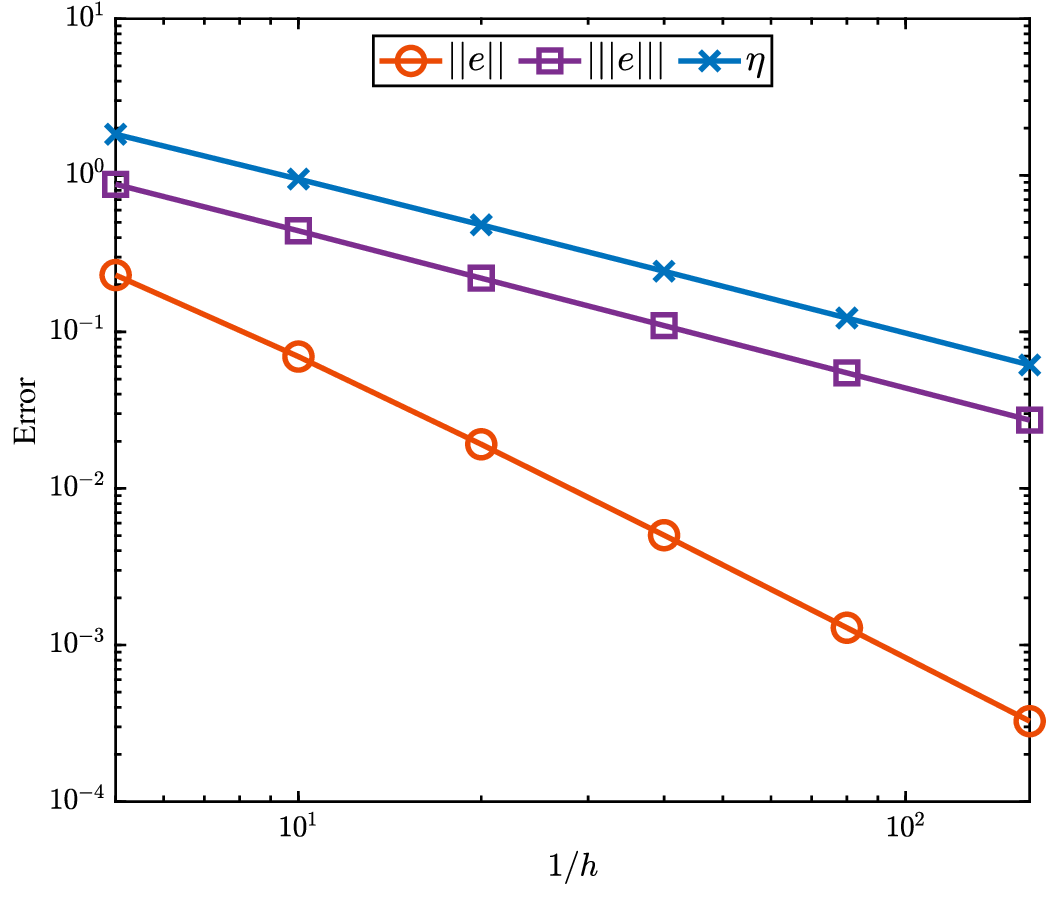}
    \includegraphics[width=0.41 \textwidth]{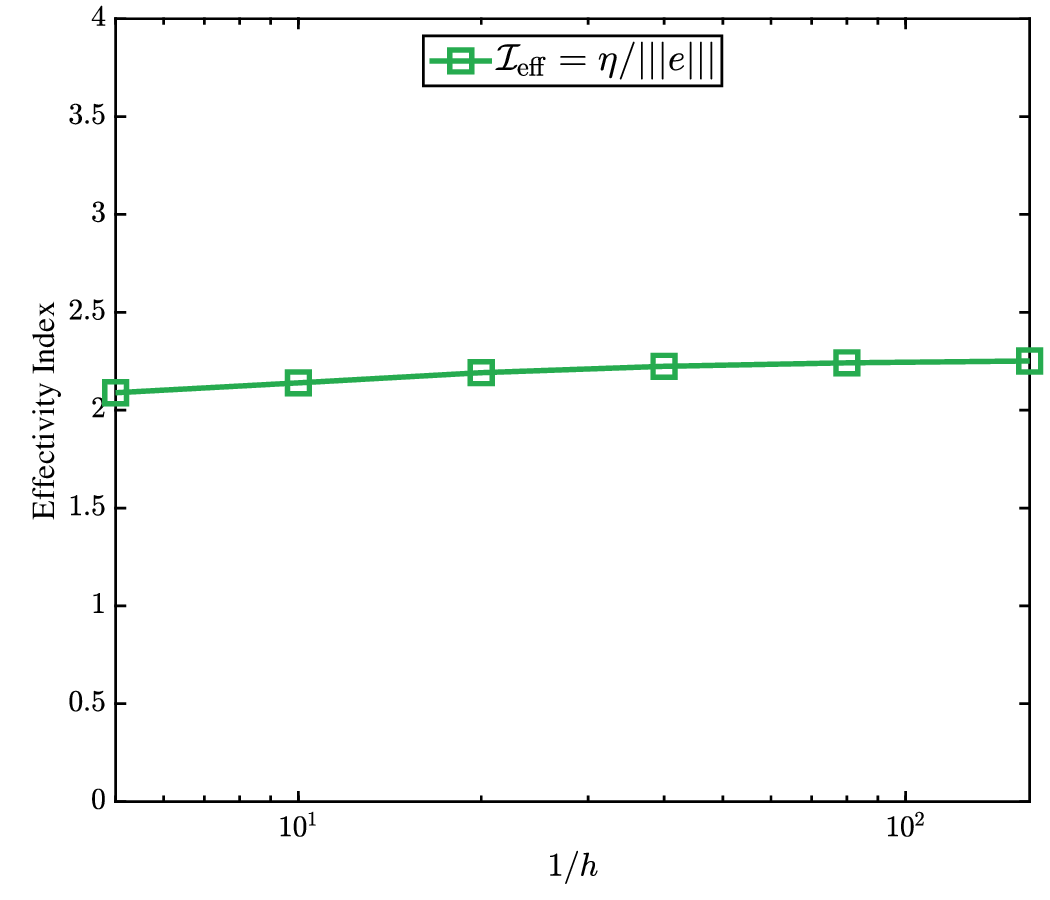}
    \caption{Results of example~\ref{example-2}. Left: The error convergence rates of the $L^2$ norm, the stabilized norm and the APEE for the diffusion-dominated problem. Right: The global effectivity index for the diffusion-dominated problem. Top: OSGS method; bottom: ASGS method.}
    \label{ex2errcon}
\end{figure}
\begin{figure}[h]
    \centering
    \includegraphics[width=0.41 \textwidth]{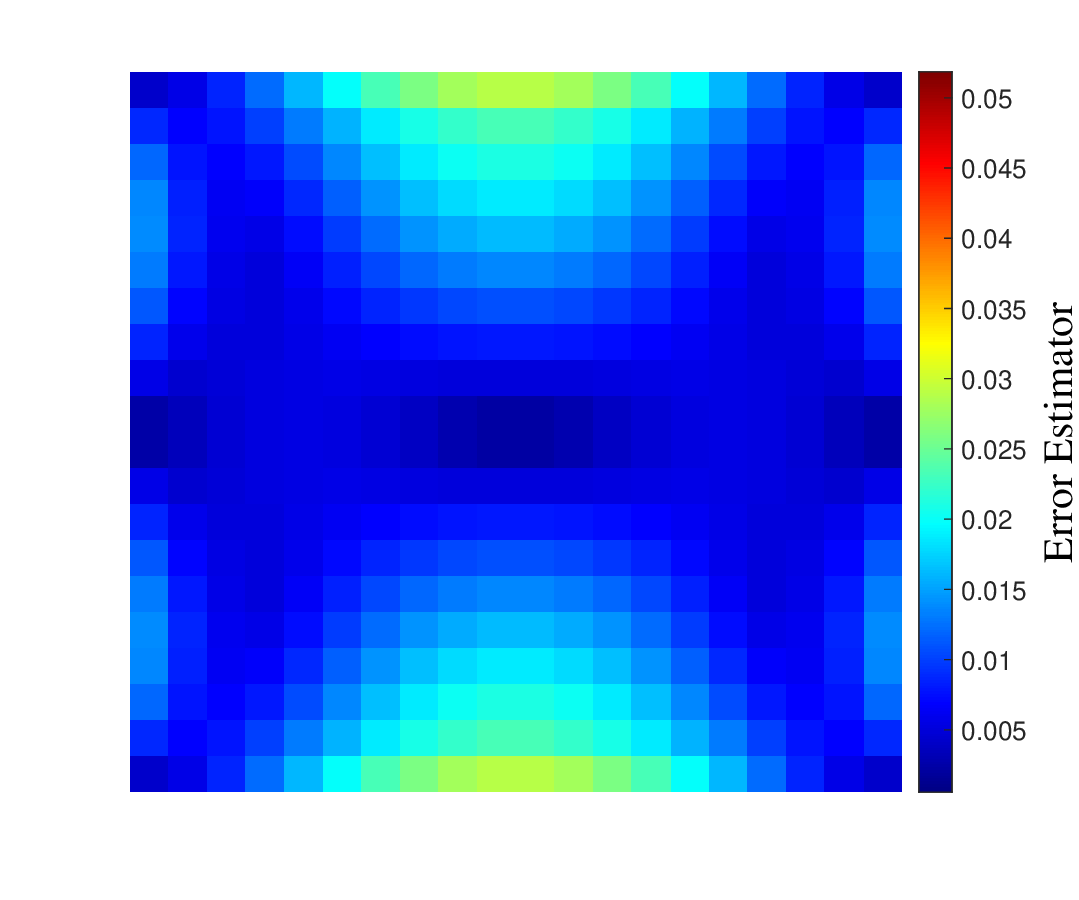}
    \includegraphics[width=0.41 \textwidth]{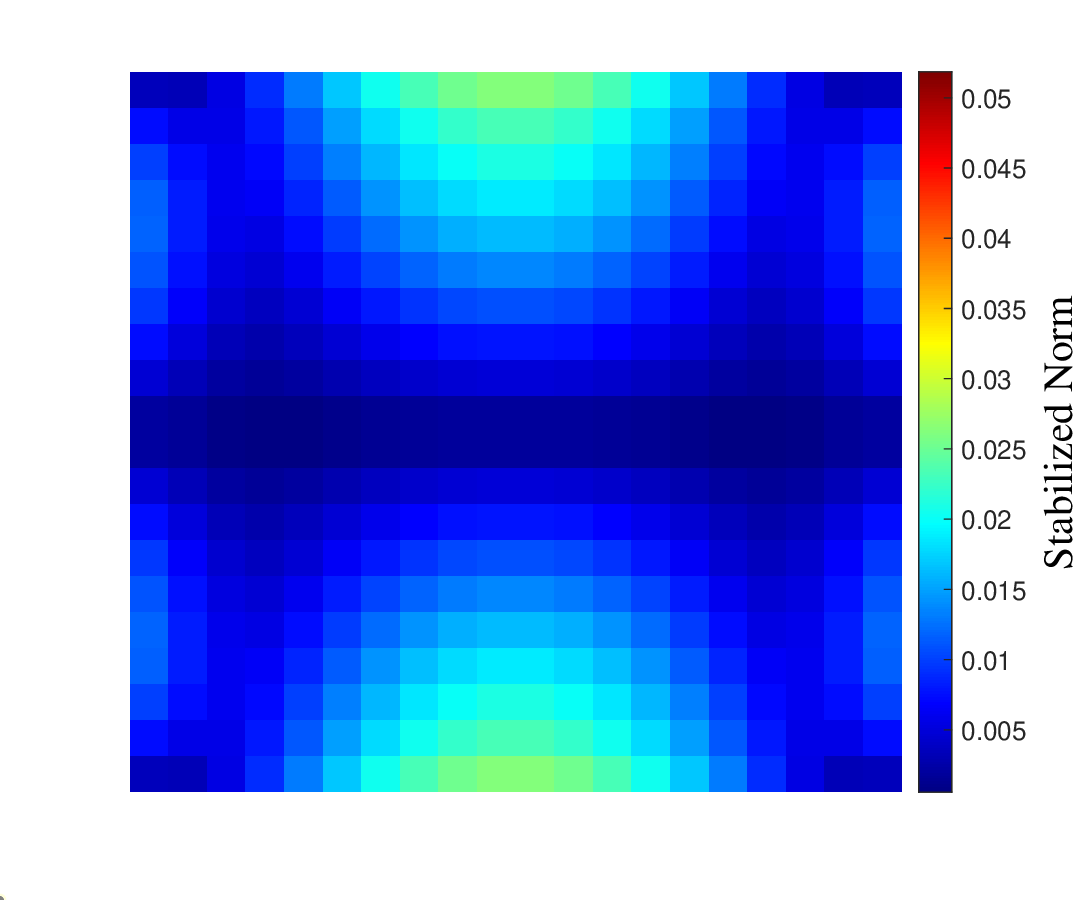}
    \includegraphics[width=0.41 \textwidth]{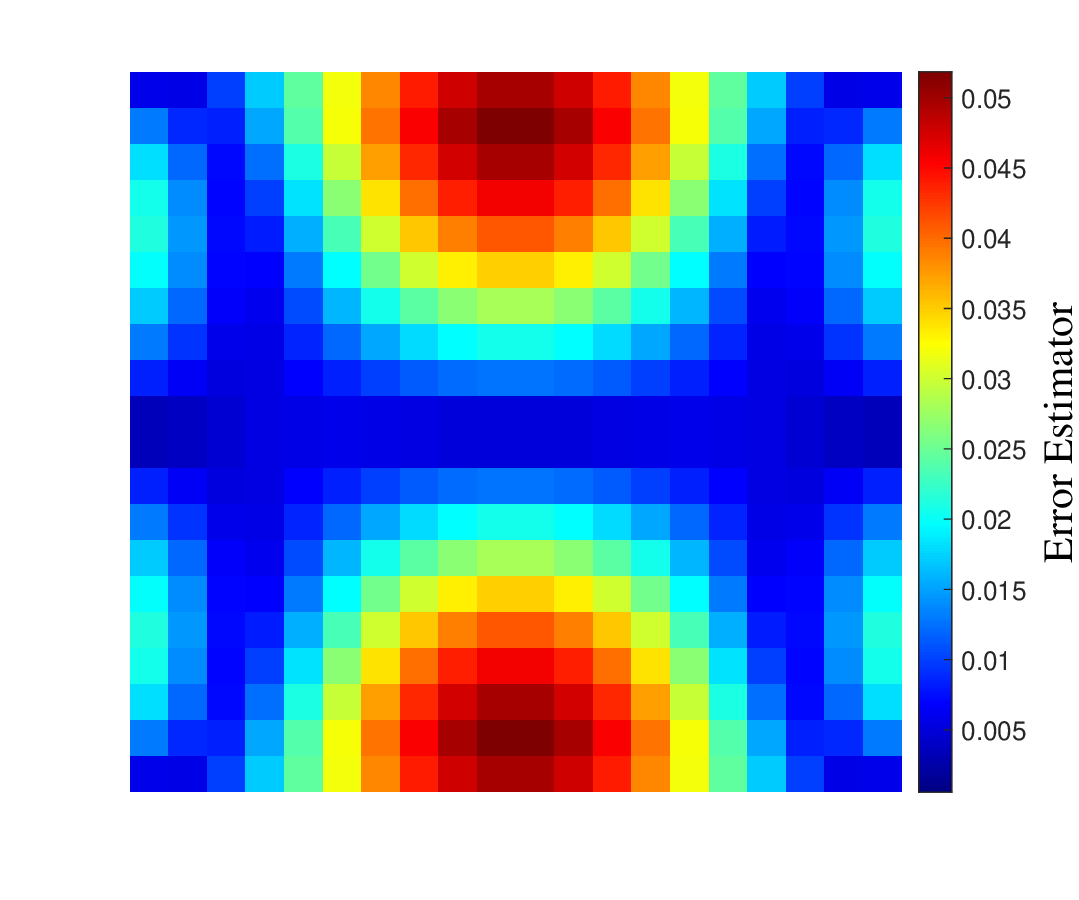}
    \includegraphics[width=0.41 \textwidth]{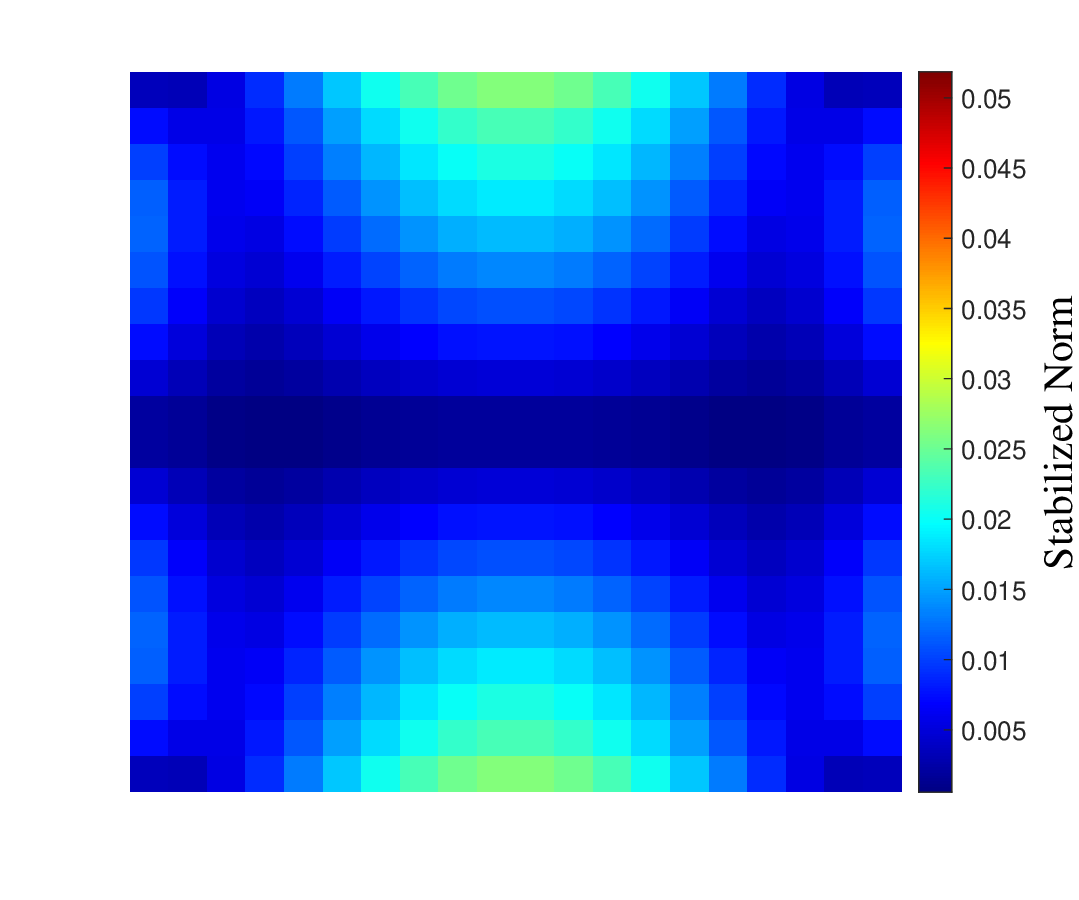}
    \caption{Results of example~\ref{example-2}. The figures on the left present the SGSs error estimator and on the right the stabilized norm error contribution. Top: OSGS method; bottom: ASGS method.}
    \label{ex2etaplot}
\end{figure}

\subsection{Strong boundary layer problem}\label{example-3}

This numerical example serves to analyze the performance of the APEE for a strong boundary layer problem. We have used the same square domain already employed in the previous examples. We have chosen $k =10^{-3}, a=[1,1]$, $s=1$, and the exact solution is given as 
\begin{equation*}
    u(x,y)=\left(x-\frac{e^{-(1-x)/k} - e^{-1/k}}{1-e^{-1/k}}\right) y\left(1-y\right)
\end{equation*}
We have again employed the concept of manufactured solutions to determine the forcing function $f$. Here, we analyze again the rate of convergence of the stabilized method and the behavior of the APEE. 

In Fig.~\ref{ex3errcon} (top), it can be observed that the APEE based on the OSGS method once again converges with the same rate as the stabilized error norm. The effectivity index converges to $1$, meaning that the APEE accurately captures the true global error. We conclude that the OSGS-based error estimator is suitable for solutions involving boundary layers. Fig.~\ref{ex3errcon} (bottom) shows the error convergence plots of the ASGS method and the effectivity index, which converges to $1.6$ for a very fine mesh. Fig.~\ref{ex3etaplot} presents the local error contributions of both the error estimator and the stabilized error norm for the OSGS and ASGS methods, computed on a uniform $100 \times 100$ quadrilateral mesh. In Fig.~\ref{ex3sfsecplo}, the strong boundary layer can be observed by the contour plot of the scalar field and the plot of the solution along the line $y=0.5$. Oscillations occur using the standard Galerkin method and are magnified near the boundary layer. It can also be noticed that there are very small overshoots near the boundary layer when using the stabilized OSGS method, which should be expected, since stabilized FE methods provide globally stable solutions, but are not monotone. It is known that local oscillations near boundary layers are in general stronger using the OSGS formulation than the ASGS method, and this causes the error to spread to the interior of the domain, as it is observed in this and the following example (see Fig.~\ref{ex3etaplot} and Fig.~\ref{ex4etaplot}, described later). In any case, these local oscillations could be removed using a shock-capturing technique. The results of Fig.~\ref{ex3sfsecplo} have been obtained by using a uniform mesh of $30 \times 30$ quadrilateral elements.

\begin{figure}[h]
    \centering
    \includegraphics[width=0.41 \textwidth]{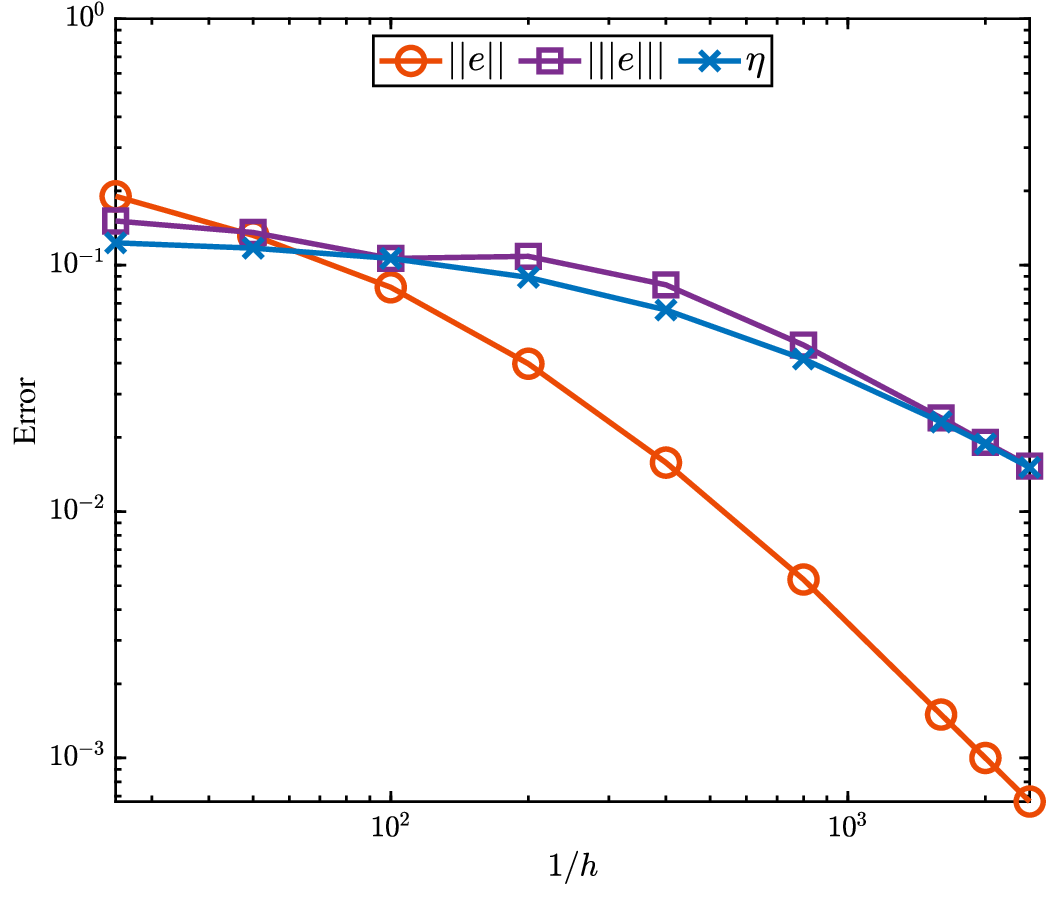}
    \includegraphics[width=0.41 \textwidth]{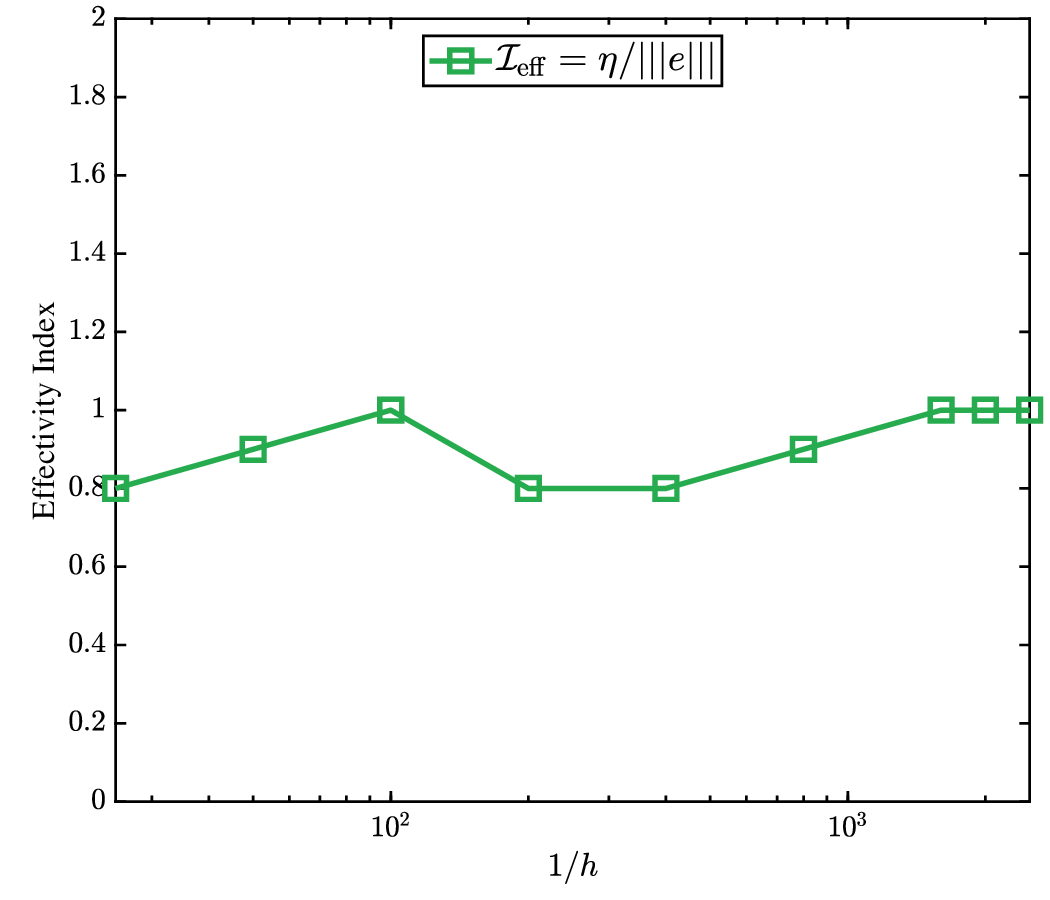}
    \includegraphics[width=0.41 \textwidth]{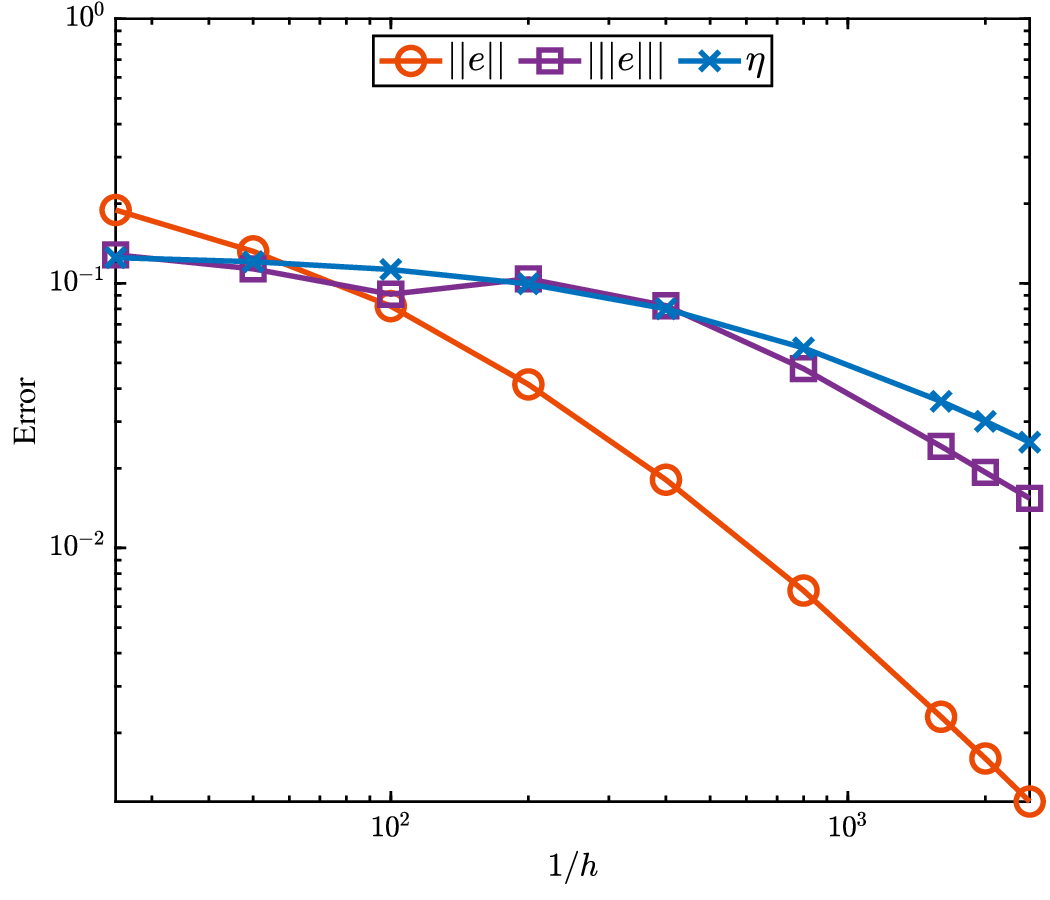}
    \includegraphics[width=0.41 \textwidth]{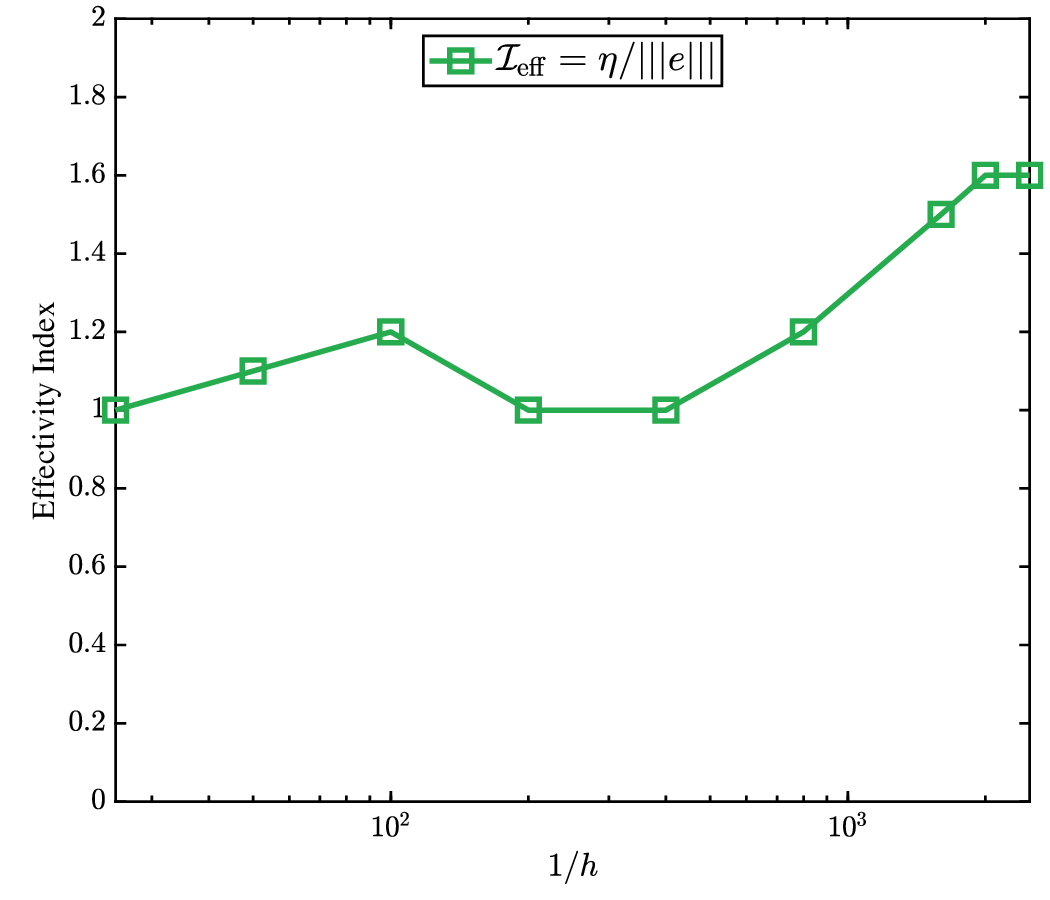}
    \caption{Results of example~\ref{example-3}. Left: The error convergence rate. Right: The global effectivity index. Top: OSGS method. Bottom: ASGS method.}
    \label{ex3errcon}
\end{figure}

\begin{figure}[h]
    \centering
    \includegraphics[width=0.41 \textwidth]{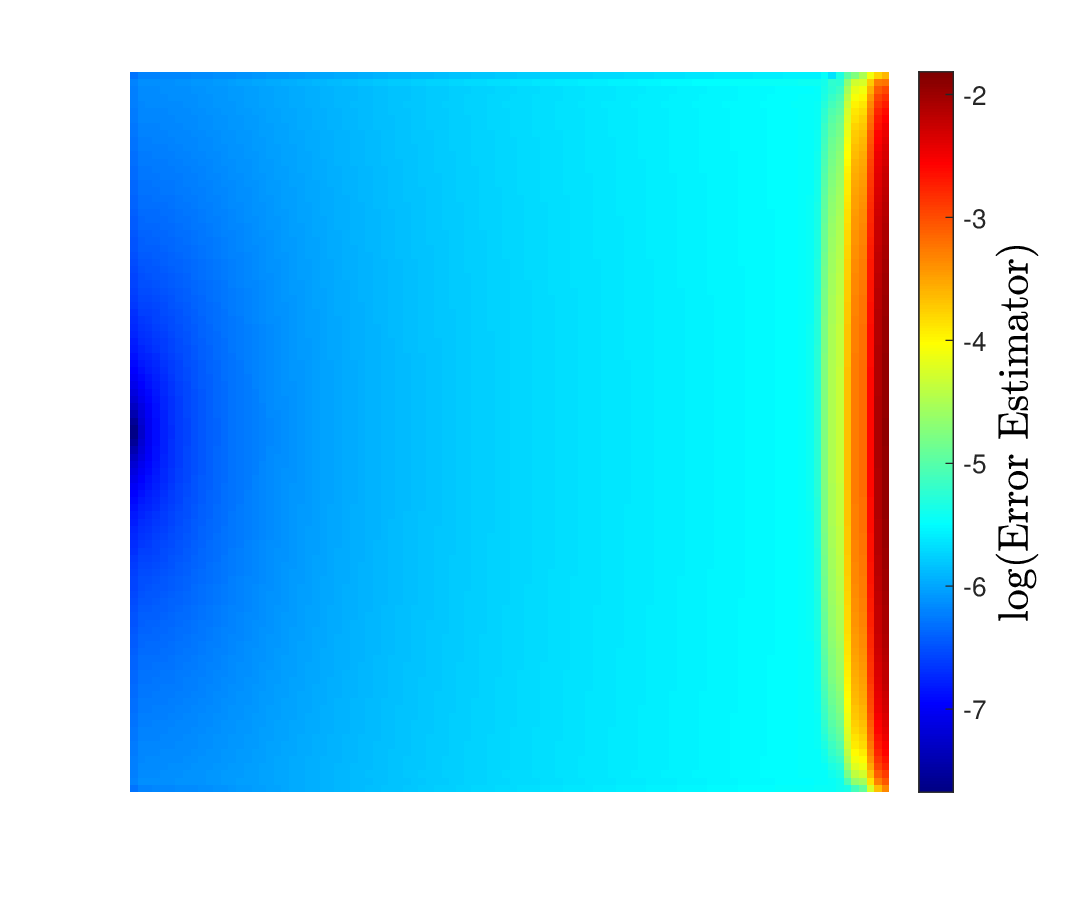}
    \includegraphics[width=0.41 \textwidth]{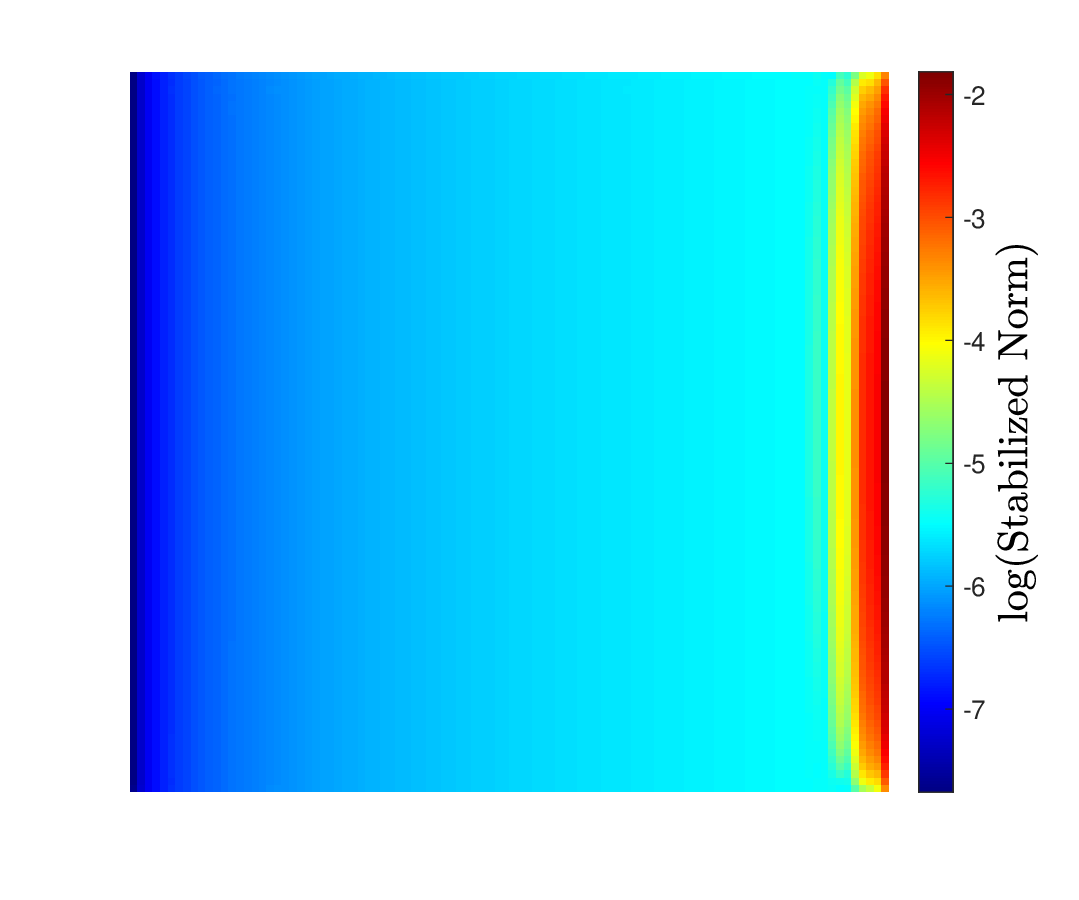}
    \includegraphics[width=0.41 \textwidth]{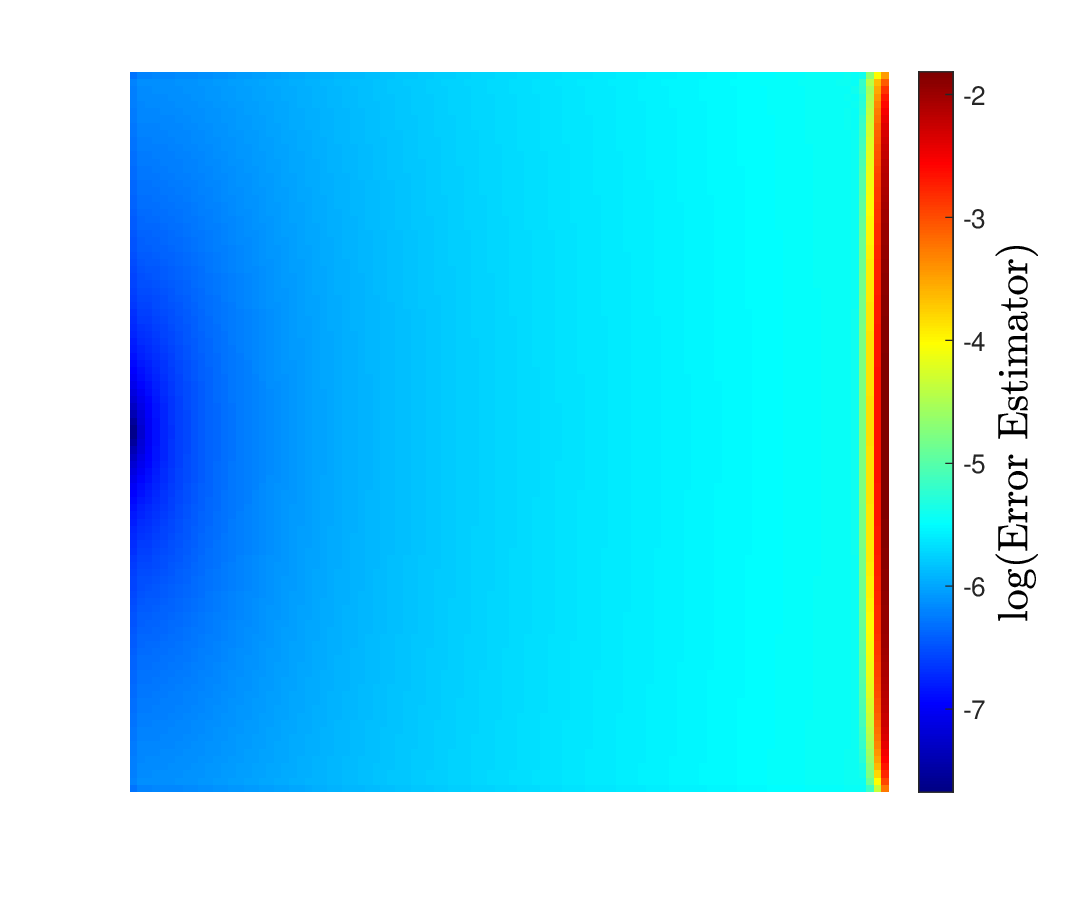}
    \includegraphics[width=0.41 \textwidth]{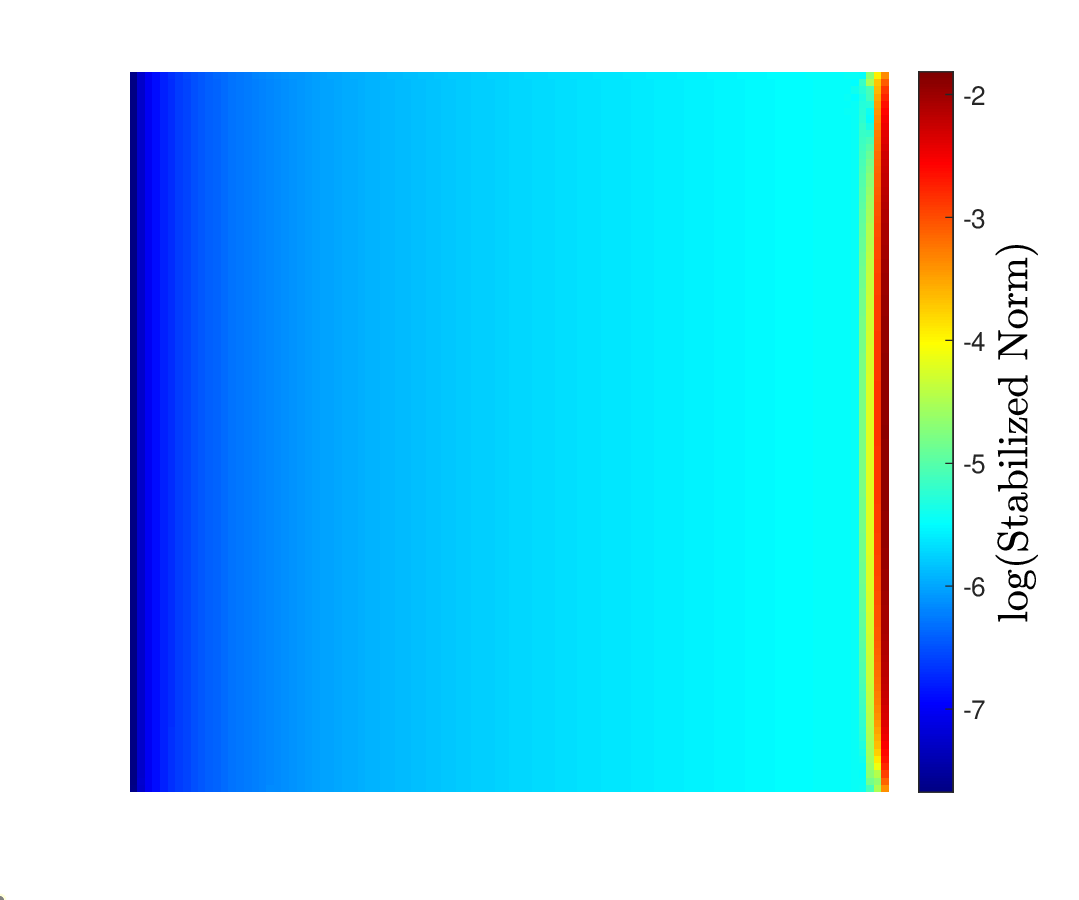}
    \caption{Results of example~\ref{example-3}. The figures on the left present the SGSs error estimator and on the right the stabilized norm error contribution. Top: OSGS method; bottom: ASGS method.}
    \label{ex3etaplot}
\end{figure}

\begin{figure}[h]
    \centering
    \includegraphics[width=0.41 \textwidth]{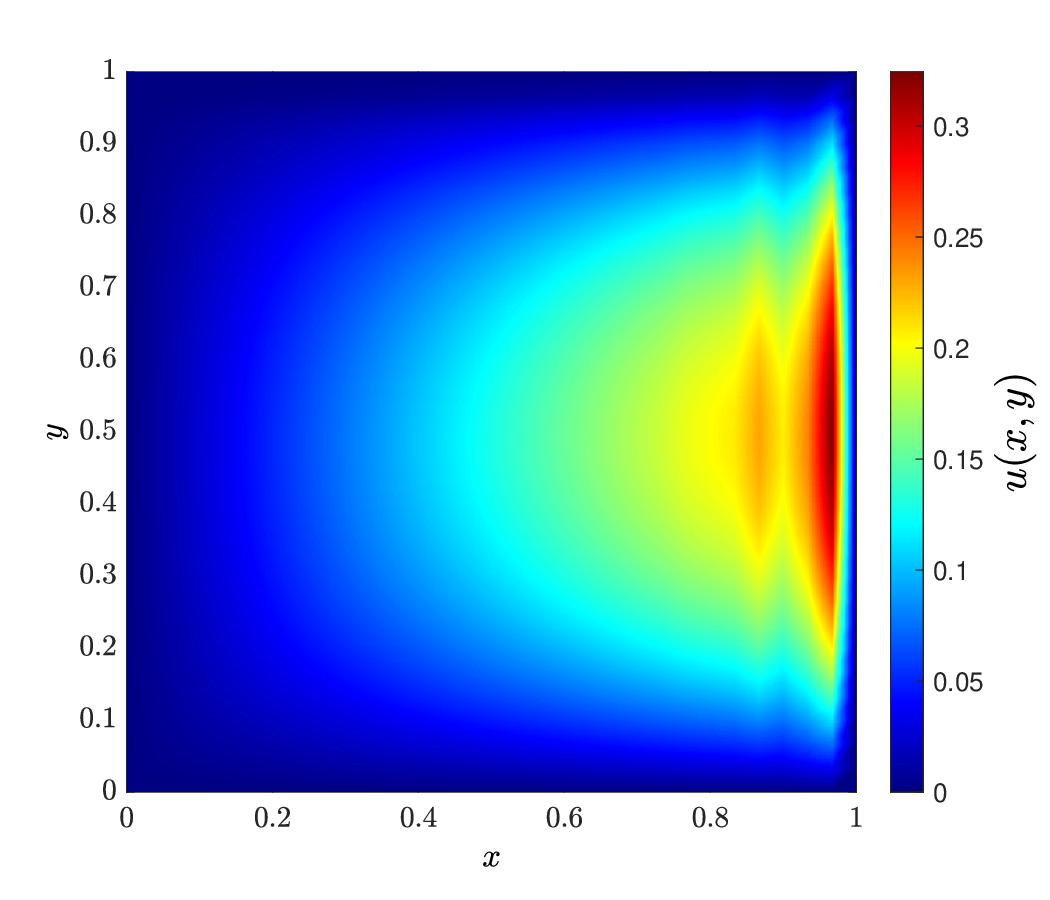}
    \includegraphics[width=0.41 \textwidth]{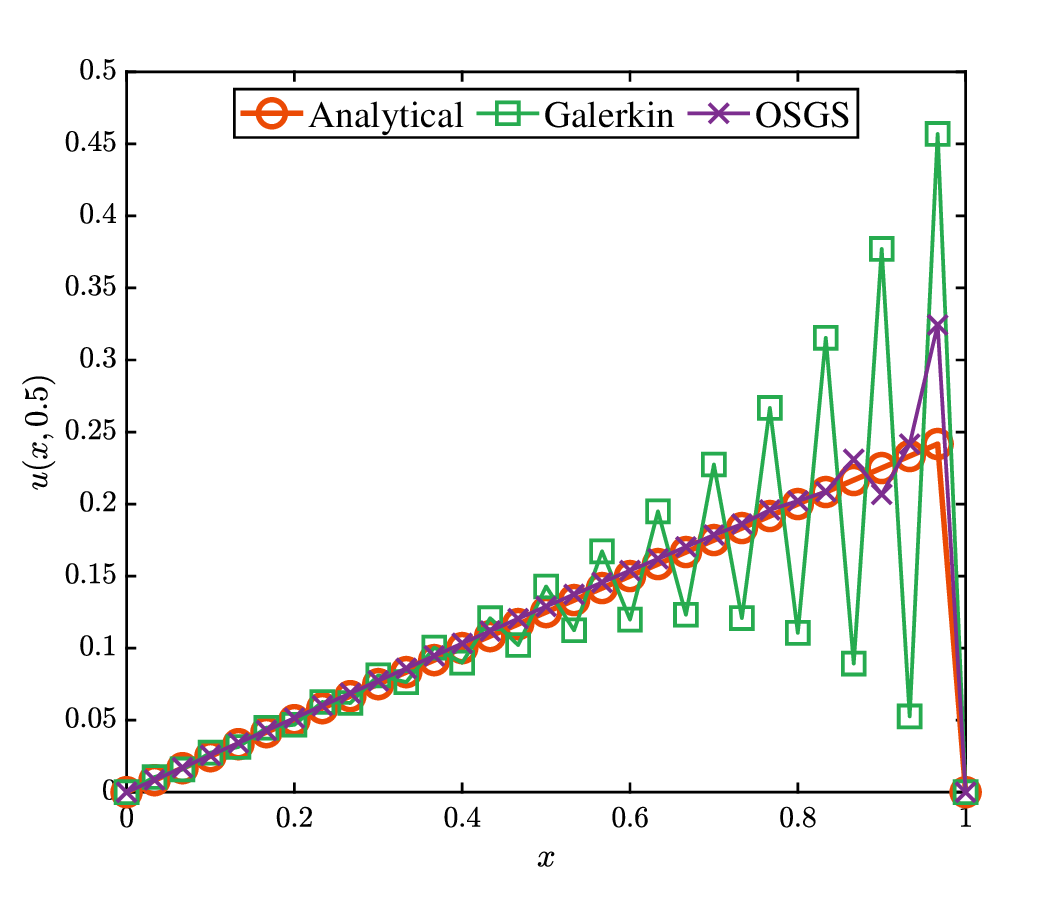}
    \caption{Results of example~\ref{example-3}. Left: Numerical solution of the scalar field. Right: Analytical, Galerkin, and OSGS solution along the line $y=0.5$.}
    \label{ex3sfsecplo}
\end{figure}

\subsection{L-shaped domain}\label{example-4}

As a final example, we study a boundary value problem on an L-shaped domain, as described in \cite{hauke2023review,john2000numerical}. We impose homogenous Dirichlet boundary conditions along the entire boundary. We choose the forcing function $f$ as
\begin{equation*}
    f=100r(r-0.5)(r-\sqrt{2}/2)
\end{equation*}
with 
\begin{equation*}
    r(x,y)=\sqrt{(x-0.5)^2+(y-0.5)^2}
\end{equation*}
The parameters in \eqref{modprob} are chosen such that the diffusion coefficient is taken as $k =10^{-6}$, the advection velocity to be $a=[1,3]$ and the reaction coefficient to be $s=1$. The computational domain is described by:
\begin{equation*}
    \Omega := \left(0,1\right) \times \left(0,1\right)\backslash\left[0.5,1\right]\times\left[0,0.5\right]
\end{equation*}
We use again a mesh of bilinear quadrilaterals to discretize this computational domain. We begin with 48 elements and successively refine the mesh by dividing each element into four. We exemplarily compare the solution using the ASGS and the OSGS methods in Fig.~\ref{ex4sfplo}. Since we do not know the exact solution of the problem, we compare the local error estimator contributions of the OSGS and the AGS methods, see Fig.~\ref{ex4etaplot}. These results were obtained using a uniform mesh consisting of 768 ($= 0.75 \cdot 32^2$) elements. We can observe the larger error on the boundary and within a shear layer in the direction of the convection originating from $(0.5,0.5)$. The boundary layer can be observed at $y=1$ in both stabilized methods. The presence of a singularity at the re-entrant corner $(0.5,0.5)$ causes the solution $u \notin H^2(\Omega)$. Therefore, we avoid computing the stabilized error norm using a fine reference solution. Instead, we refer to \cite{hauke2023review}, where the effectivity indices were assessed using $L^1$ and $L^2$ norms.

\begin{center}
    \begin{figure}[h]
        \centering
        \includegraphics[width=0.41 \textwidth]{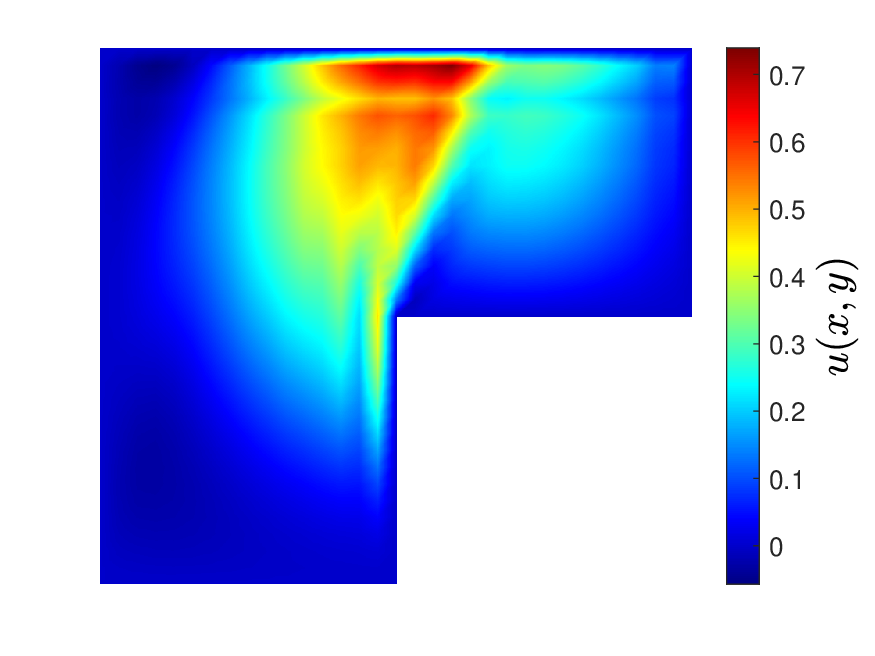}
        \includegraphics[width=0.41 \textwidth]{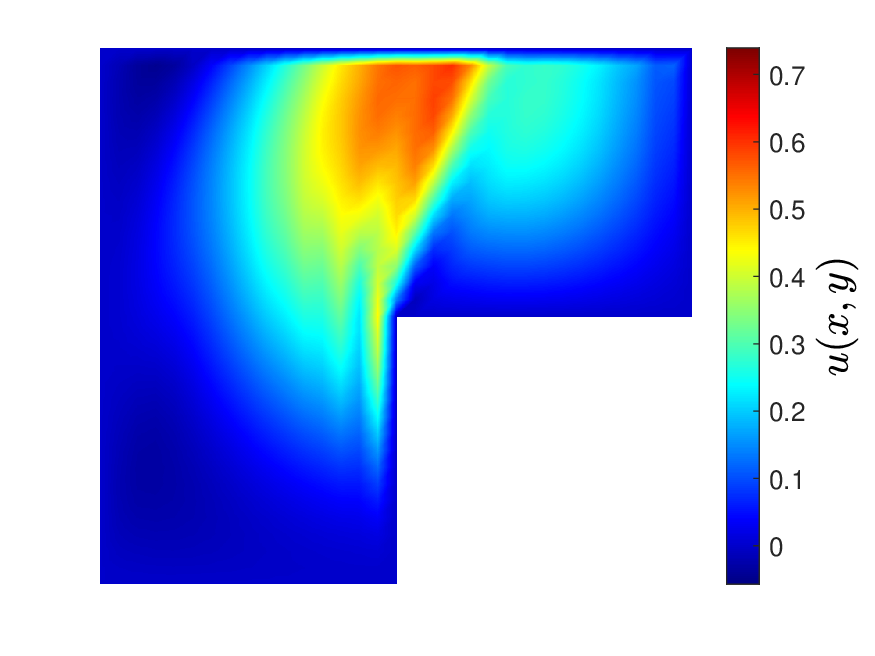}
      \caption{Numerical solution of the scalar field computed in example~\ref{example-4}. Left: OSGS method. Right: ASGS method.}
        \label{ex4sfplo}
    \end{figure}
\end{center}

\begin{center}
    \begin{figure}[h]
        \centering
        \includegraphics[width=0.41 \textwidth]{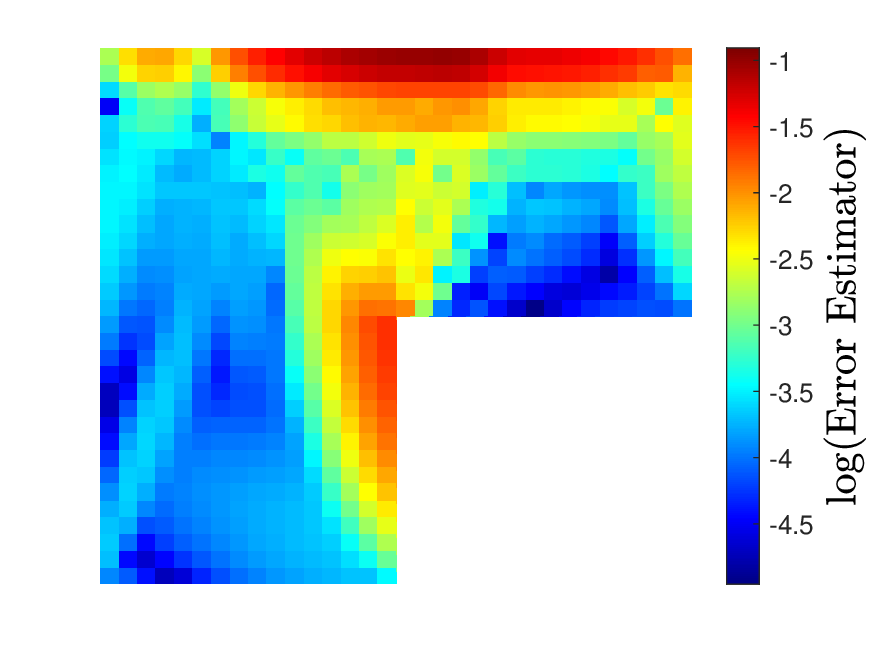}
        \includegraphics[width=0.41 \textwidth]{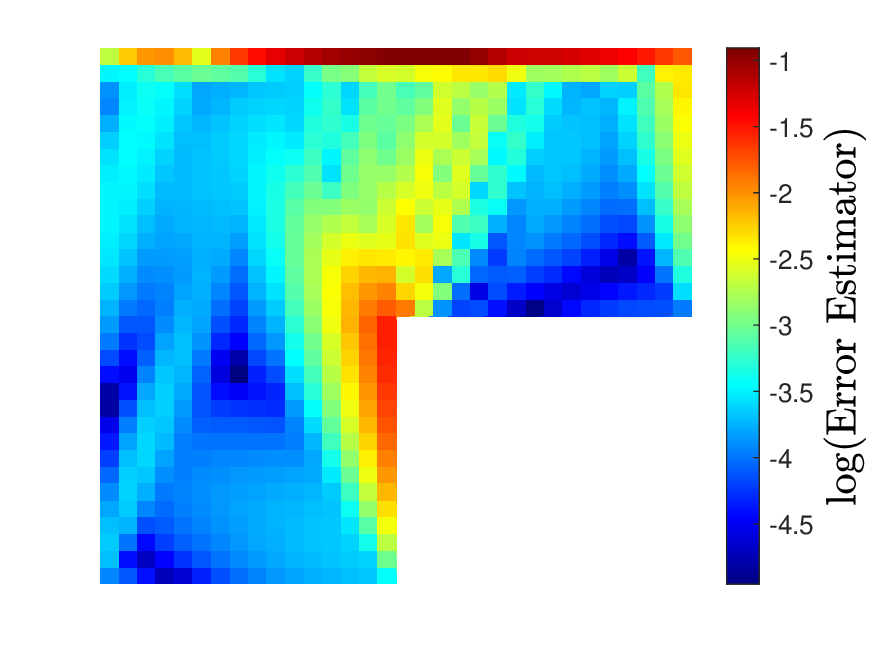}
     \caption{Results of example~\ref{example-4}. The local error estimator contribution for the L-shaped domain. Left: OSGS method. Right: ASGS method.}
        \label{ex4etaplot}
    \end{figure}
\end{center}

\section{Conclusion}

In this paper, we have proposed an a posteriori error estimate for methods falling in the VMS framework. In summary, it is an estimate {\em in the stabilized norm} of the formulation, and it states that the error behaves as {\em a scaled norm of the SGSs}, the scaling factors being the inverse of the stabilization parameters. This error estimate incorporates the contribution of two parts: the SGSs in the element interiors and on the edges. In the diffusion-dominated regime, the SGSs on the edges play a vital role in error estimation.
The error estimation is carried out by post-processing the FE solution, i.e., it falls in the category of explicit a posteriori error estimates. This method of error computation results in a lower computational cost than the APEEs based on solving an additional auxiliary problem. 

We have presented a theoretical justification for the proposed a posteriori error estimate. We have started proving sharp upper and lower bounds for the error in a norm that is not the stabilized one~(Section~\ref{sec:verf}), and then, we have proceeded with the analysis in the stabilized norm, although the estimate depends on an interpolation error of the exact solution and on an unprovable assumption~(Section~\ref{sec:john-novo}). Although not fully satisfactory from the theoretical point of view, this analysis provides a sound basis for our proposal. 

In the context of VMS methods, we have focused our attention on the OSGS formulation, with the SGSs orthogonal to the FE space, and the analysis has been carried out for this method. The numerical examples have shown that the performance of the proposed a posteriori error estimate is excellent, in all cases, the effectivity indices always converge to a value very close to $1$, which is in contrast to the ASGS method.

\section*{Acknowledgements}

R. Codina gratefully acknowledges the support received from the ICREA Acad\`emia Program, from the Catalan Government.
S. A. Khan acknowledges the financial support received from the Program Contract between CIMNE and the Catalan Government for the period 2021-2024.

\bibliographystyle{elsarticle-num}

\end{document}